\documentclass[11pt]{scrartcl}

\usepackage{authblk}

\usepackage{setspace}

\usepackage{subfigure}
\usepackage{amsfonts}
\usepackage{amssymb}
\usepackage{amsmath,mathrsfs}
\usepackage{amsthm}
\usepackage{latexsym}
\usepackage{layout}
\usepackage{graphics}
\usepackage{enumerate}
\usepackage{epsfig}
\usepackage{algorithmicx}
\usepackage{algpseudocode}
\usepackage[usenames,dvipsnames]{xcolor}	
\usepackage{algorithm}

\DeclareMathAlphabet{\mathcal}{OMS}{cmsy}{m}{n} 

\newcommand{\R}{{\mathbb{R}}}
\newcommand{\N}{{\mathbb{N}}}

\newcommand{\Ss}{{\mathcal{S}}} 
\newcommand{\NN}{{\mathcal{N}}}
\newcommand{\SLP}{{\mathcal{SP}}}
\newcommand{\F}{{\mathcal{F}}}
\newcommand{\br}{{\bar{r}}}
\newcommand{\ldef}{\mathrel{\mathop:}=}
\newcommand{\rdef}{=\mathrel{\mathop:}}
\newcommand{\lone}{\ell_1}
\newcommand{\ltwo}{\ell_2}
\newcommand{\lp}{\ell_p}
\newcommand{\subonetoN}{\subset\{1,\ldots,N\}}
\newcommand{\onetoN}{1,\ldots,N}
\newcommand{\vectornorm}[1]{\left\|#1\right\|}
\newcommand{\kappap}{\mathcal{\kappa}_p}
\newcommand{\irwl}{{IRW$\lone$ }}

\DeclareMathOperator*{\argmin}{arg\,min}
\DeclareMathOperator*{\supp}{supp}
\DeclareMathOperator*{\sign}{sign}

\newtheorem{definition}{Definition}

\newtheorem{remark}{Remark} 
\newtheorem{lemma}{Lemma} 
\newtheorem{theorem}{Theorem}

\begin{document}

\title{Damping Noise-Folding and Enhanced Support Recovery in Compressed Sensing}
\subtitle{Extended Technical Report}

\author{Marco Artina\thanks{Marco.Artina@ma.tum.de}}
\author{Massimo Fornasier\thanks{Massimo.Fornasier@ma.tum.de}}
\author{Steffen Peter\thanks{Steffen.Peter@ma.tum.de}}
\affil{\footnotesize{Technische Universit\"at M\"unchen, Fakult\"at Mathematik, Boltzmannstrasse 3 D-85748, Garching bei M\"unchen, Germany}}

\maketitle

\begin{abstract}
{The practice of compressed sensing suffers importantly in terms of the efficiency/accuracy trade-off when acquiring noisy signals prior to measurement. It is rather common to find results treating the noise affecting the measurements, avoiding in this way to face the so-called
\textit{noise-folding} phenomenon, related to the noise in the signal, eventually amplified by the measurement procedure. In this paper, we present two new decoding procedures,  combining $\ell_1$-minimization followed by either a regularized selective least $p$-powers or an iterative hard thresholding, which not only are able to reduce this component of the original noise, but also have enhanced properties in terms of support identification with respect to the sole $\ell_1$-minimization or iteratively re-weighted $\lone$-minimization. We prove such features, providing relatively simple and precise theoretical guarantees. We additionally confirm and support the theoretical results by extensive numerical simulations, which give a statistics of the robustness of the new decoding procedures with respect to more classical $\ell_1$-minimization and iteratively re-weighted $\lone$-minimization.}
{Noise folding in compressed sensing, support identification, $\ell_1$-minimization, selective least $p$-powers, iterative hard thresholding, phase transitions.}
\\
\\
2000 Math Subject Classification: 94A12, 94A20, 65F22, 90C05, 90C30
\end{abstract}

\section{Introduction}
Compressive sensing focuses on the robust recovery of nearly sparse vectors from the minimal amount of measurements obtained by a randomized linear process. So far, a vast literature appeared considering problems where deterministic or random noise is added after the measurement process, while it is not strictly related to the signal. One typically considers model problems of the type
\begin{equation}
y = Ax + w
\label{eq:enc}
\end{equation}
where $x\in\R^N$ is a nearly sparse vector, $A\in\R^{m\times N}$ is the linear measurement matrix, $y\in\R^m$ is the result of the measurement, and $w$ is a white noise vector affecting the measurements. However, in practice it is very uncommon to have a signal detected by a certain device, totally free from some external noise. Therefore, it is reasonable to consider the more realistic model
$$
y = A(x+n) + w,
$$
instead of \eqref{eq:enc} where $n\in\R^N$ is the noise on the original signal.\\

The recent work \cite{castro11,baraniuk09} shows how the measurement process actually causes the \textit{noise-folding phenomenon}, which implies that the variance of the noise on the original signal is amplified by a factor of $\frac{N}{m}$, additionally contributing to the measurement noise, playing to our disadvantage in the recovery phase. We refer also to the  papers \cite{WWS13,WWS14}, which seem to have independently reconsidered this problem more recently.
More formally, if we add to the signal $x$ a noise vector $n$, whose entries have normal distribution $\NN(0,\sigma_n)$,
the measurement $y$ given by
\begin{equation} \label{vectornoise}
y = A(x+n),
\end{equation}
can be considered equivalently obtained by a measurement procedure of the form \eqref{eq:enc} (possibly with another measurement matrix $A$ of equal statistics) where now the vector $w$ is composed by i.i.d. Gaussian entries with distribution $\NN(0,\sigma_w)$ and $\sigma_w^2=\frac{N}{m}\sigma_n^2$.\\
In this stochastic context, the so-called Dantzig selector has been analyzed in \cite{candes07} showing that the recovered signal $x^*$ from the measurement $y$ fulfills the following nearly-optimal distortion guarantees, under the assumption that $A$ satisfies the so-called restricted isometry property (compare Definition~\ref{def:rip}): 
\begin{equation}\label{dantzig}
\|x-x^*\|^2_{\ltwo} \leq C^2 \cdot 2\log{N} \cdot \left( \sigma_w^2 + \sum_{i=1}^N \min \{x_i^2,\sigma_w^2\} \right),
\end{equation}
which, for a sparse vector $x$ with at most $k$-nonzero entries, reduces to the following estimate
\begin{eqnarray*}
\|x-x^*\|^2_{\ltwo} &\leq& C^2 \cdot 2\log{N} \cdot \left( (1+k) \sigma_w^2\right)\\
&=&  C^2 \cdot 2\log{N} \cdot \left( (1+k) \frac{N}{m} \sigma_n^2 \right).
\end{eqnarray*}
Therefore, the noise-folding phenomenon may significantly reduce in practice the potential advantages of compressed sensing in terms of the trade-off between robustness and efficient compression (here represented by the factor $\frac{N}{m}$), with respect to other more traditional subsampling encoding methods \cite{DLB12}.\\
In the following, we wish to focus on two fundamental consequences of noise folding: the loss of accuracy in the recovery of the large entries of the original vector $x$, and the correct detection of their index support.\\
To this end, let us introduce for $r>\eta > 0$ and $1 \leq k < m$ the class of \emph{sparse vectors affected by bounded noise},  
\begin{equation}
\Ss_{\eta,k,r}^p\ldef\left\{x\in\R^N\big|\#S_r(x)\leq k \text{ and} \sum_{i\in(S_r(x))^c}|x_i|^p\leq\eta^p\right\}, \quad 1\leq p\leq 2,
\label{eq:Setakr}
\end{equation}
where $S_r(x)\ldef\left\{i\in\{\onetoN\}\right|\left.|x_i|>r \right\}$ is the index support of the large entries exceeding in absolute value the threshold $r$. This class contains all vectors for which at most $1 \leq k < m$ large entries exceed the threshold $r$ in absolute value, while the $p$-norm of the other entries stays below a certain noise level. To highlight the concrete modeling potential of such a class, in Figure \ref{fig:compvec} we show typical examples of real-life signals from Asteroseismology falling into it, for some choice of the parameters $\eta,k,r$.

\begin{figure}[!htb]
\begin{center}
\subfigure[Power density spectrum of the red giant candidate CoRoT-101034881 showing a frequency pattern with a regular spacing~\cite{astero1}.]{ 
\includegraphics[width=0.39\textwidth]{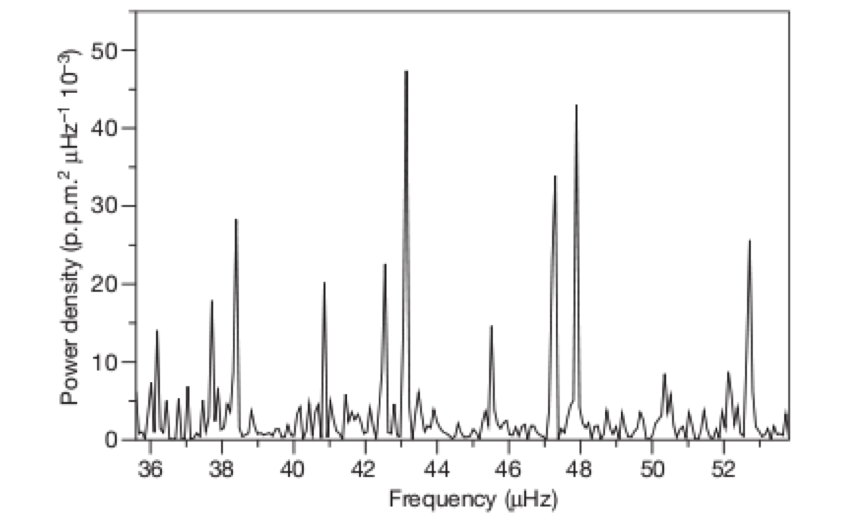}}
\hspace{1cm}\subfigure[Amplitude spectra of solar oscillations measured by the VIRGO instrument on SOHO. The figure shows only a portion of the solar oscillation spectrum~\cite{astero2}.]{
\includegraphics[width=0.5\textwidth]{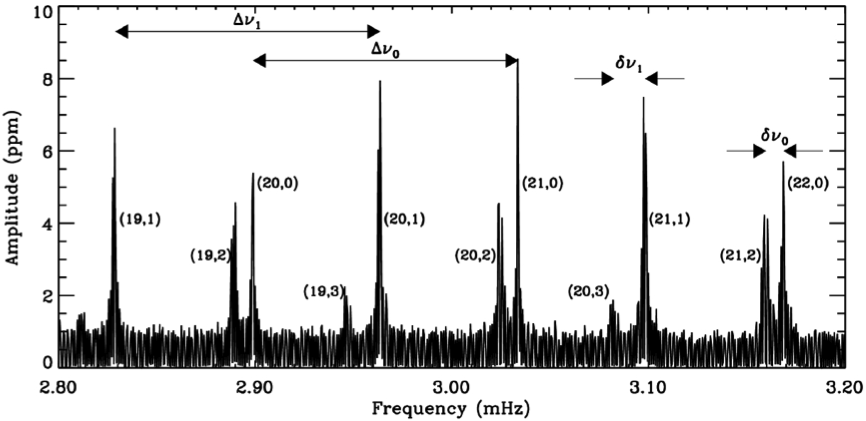}}
\end{center}
\caption{Examples of signals out of the class defined in \eqref{eq:Setakr}.}
\label{fig:compvec}
\end{figure}

Let us now informally explain how the class $\Ss_{\eta,k,r}^p$ can be crucially used to analyze the effects of noise folding in terms of support recovery depending on the parameters $\eta, r, k$. Therefore, assume now that $x$ is a sparse vector with at most $k$ nonzero entries exceeding the threshold $r$ in absolute value and $x+n \in \Ss_{\eta,k,r}^2$ in expectation (notice that  we specified here $p=2$). By the statistical equivalence mentioned above of the model \eqref{vectornoise} and \eqref{eq:enc}, we infer that the recovered vector $x^*$ by means of the Dantzig selector will fulfill the following error estimate:

\begin{align}\label{datzsparse}
\|x-x^*\|^2_{\ltwo} & \leq C^2 \cdot 2\log{N} \cdot \left( (1+k) \frac{N}{m}\sigma_n^2\right) \nonumber \\
& \leq C^2 \cdot 2\log{N} \cdot \left( (1+k) \frac{N}{m}\frac{\eta^2}{N-k}\right),
\end{align}
where the last inequality follows by the requirement 
$$
(N-k)\sigma_n^2 = \mathbb E \left (\sum_{i\in(S_r(x))^c}|n_i|^2 \right )\leq \eta^2,
$$
considering here $S_r(x)$ as previously defined in \eqref{eq:Setakr}. Since we assume that $k \ll N$ and $\frac{k+1}{m}\leq1$, the right-hand-side of \eqref{datzsparse} can be further bounded from above by
$$
C^2 \cdot 2\log{N} \cdot \left( (1+k) \frac{N}{m}\frac{\eta^2}{N-k}\right) \leq C_1^2 \cdot 2\log{N} \cdot \eta^2.
$$ 
It easily follows (and we will use similar arguments below for different decoding methods) that a sufficient condition for the identification of the support of $x$, i.e., $\supp(x) \subset \supp(x^*)$, is
$$
C_1^2 \cdot 2\log{N} \eta^2 < r^2,
$$
or, equivalently, 
$$
\eta < \frac{r}{C_1\sqrt{2\log{N}}}.
$$ 
Notice that such a sufficient condition on $\eta$ actually implies a rather large gap between the significant entries of $x$ and the noise components in $n$.
Hence, it is of utmost practical interest to understand how small this gap is actually allowed to be, i.e., how large $\eta$ can be relatively to $r$, for the most used recovery algorithms in compressed sensing (not only the Dantzig selector) to be able to have both support identification and a good approximation of the significant entries of $x$. 

An approach to control the noise folding, is proposed in \cite{castro11}. In this case, one may tune the linear measurement process in order to a priori filter the noise. However, this strategy requires to have a precise knowledge of the noise statistics and to design proper filters. Other related work \cite{HBCN09,HBCN12,HCN11} addresses the problem of designing adaptive measurements, called {\it distilled sensing}, in order to detect and locate the signal within white noise.
In this paper, we shall follow a \textit{blind-to-statistics} approach (although, as recalled below, we do not consider the scenario of impulsive noise), which does not modify the non-adaptive measurements, and, differently from the Dantzig selector analysis in \cite{candes07}, we restrict ourself to a purely deterministic setting.
\\

First of all, we show that, unfortunately, the classical $\lone$-minimization, but also the iteratively re-weighted $\ell_1$-minimization \cite{candesIRWL1,needellIRWL1}, considered one of the most robust in the field, easily fails in both the tasks mentioned above as soon as $\eta\approx r$. The deep reason of this failure is the lack of selectivity of these algorithms, which are designed to promote not only the sparsity of the signal $x$ but also of the recovered noise. This principle drawback affects necessarily also recent modifications of re-weighted $\ell_1$-minimization as appearing in \cite{KXAH10}.
This has the consequence,  as a sort of balancing principle, that the miscomputed components of the noise $n$, if it is not originally of impulsive nature, blow up the inaccuracy in the detection and approximation of $x$. (Notice that the problem of separating sparse signals within impulsive noise is much harder and it will not be considered within the scope of the present paper.)\\
To overcome these difficulties of these popular methods,  we propose a new decoding procedure, combining   $\ell_1$-minimization and regularized {\it selective least $p$-powers}, which is able to reduce the noise component affecting the signal and also to enhance the support identification. In fact, for certain applications, such as radar \cite{hurast12}, the support recovery can be even more relevant than an accurate estimate of the signal values. Although the analysis and the numerical results of this new procedure greatly
outperform $\ell_1$-minimization and its re-weighted iterative version, its computational complexity becomes prohibitive in very high-dimension. For we conclude by proposing and analyzing an eventual method, based on similar principles, which combines again a warm-up step based on $\ell_1$-minimization with a nonconvex optimization realized by the well-known iterative hard thresholding~\cite{blumensath2008iterative}, and a final  correction step realized by a convex optimization. We show that this latter method performs as robustly as the previous one, but with a drastic reduction of the complexity.
\\

The paper is organized as follows. In the next section, we concisely recall the pertinent features of the theory of compressive sensing. In Section~\ref{sec:SI}, we shall describe the limitations of $\lone$-minimization when noise on the signal is present, and we mention that very similarly an analogue analysis can be performed for the iteratively re-weighted $\lone$-minimization based on the results in \cite{needellIRWL1}. Afterwards, as an alternative, we propose the linearly constrained minimization of the regularized selective $p$-potential functional, and show that certain sufficient conditions for recovery indicate significantly better performance than the one provided by $\lone$-minimization and iteratively re-weighted $\ell_1$-minimization. Within this paper, we measure the \emph{performance} of a method by its ability to identify and approximate the relevant entries of the original signal, where the \emph{relevant entries} of a signal are the ones which exceed a predefined threshold $r$. For this purpose, we introduced the class \eqref{eq:Setakr}. With respect to the notation in the introduction, this class is changing the meaning of the notation $x$ throughout the rest of the paper and makes redundant the notation $x+n$, which emphasizes the fact that the noise is directly on the signal. Section~\ref{sec:MSmin} recalls the main properties of a very robust and efficient algorithm to perform the linearly constrained minimization of the regularized selective least $p$-powers. In Section~\ref{sec:selectivity}, we address the issue of the high computational cost of the regularized selective $p$-potential optimization and propose, exploiting a similar \emph{selectivity principle}, an alternative method based on iterative hard thresholding. Finally, in Section~\ref{sec:NR}, we report the results of extensive numerical experiments, which we made to illustrate and support our theoretical guarantees, and the comparisons between all the mentioned decoding methods.




\section{Compressive sensing}\label{sec:CS}
It is possible to uniquely and robustly identify the solution $x\in\R^N$ of the linear system $Ax=y$ for an arbitrary given measurement vector $y\in\R^m$, if $A\in\R^{m\times N}$ has rank $m$, and $m=N$. However, in many applications, we may be either not able to take enough measurements, or interested in taking much fewer measurements to save costs or time, i.e., $m \ll N$. The theory of compressive sensing studies this scenario under some restrictions, and assumes that the original signal $x$ is nearly sparse. In this section, we recall concisely terms and principles of this theory, and we refer to some of the known tutorials for more details \cite{ba07-2,cawa08,fora10,fora10-1}.\\

In compressive sensing, we call the matrix $A$ the \emph{encoder} which transforms the $N$-dimensional signal $x$ into to the \textit{measurement vector} $y \in \mathbb R^m$ of dimension $m \ll N$. Further, we assume $A$ to have rank $m$ from now on in this article. In practice, we do not know $x$ and wonder if it is possible to recover it somehow robustly by an efficient nonlinear \emph{decoder} $\Delta \colon\R^m\rightarrow\R^N$. As already mentioned, the theory only works if we assume the signal $x$ to be sparse or at least compressible.
\begin{definition}[$k$-sparse vector]
Let $k\in\N^+$, $k \leq N$. We call the vector $x\in\R^N$ \emph{$k$-sparse} if 
\[x\in\Sigma_k\ldef\left\{z\in\R^N|\#\supp(z) \leq k\right\}\]
where $\supp(z)\ldef\left\{i \in \{1,\ldots,N\} | z_i\neq0 \right\}$ denotes the \emph{support} of $z$. 
\end{definition}
In applications, signals are often not exactly sparse but at least compressible. We refer to~\cite{Mallat} for more details. We define compressibility in terms of the  \emph{best $k$-term approximation error} with respect to the $\ell_p$-norm, given by
$$
\|x\|_{\ell_p} = \left ( \sum_{i=1}^N |x_i|^p \right )^{1/p}, \quad 1 \leq p < \infty.
$$
\begin{definition}[Best $k$-term approximation]
Let $x$ be an arbitrary vector in $\R^N$. We denote the \emph{best k-term approximation} of $x$ by 
\[x_{[k]}\ldef\argmin\limits_{z\in\Sigma_k}\left\|x-z\right\|_{\lp},\qquad 1\leq p <\infty,\] 
and
the respective \emph{best $k$-term approximation error} of $x$ by 
\[\sigma_k(x)_{\lp}\ldef\min\limits_{z\in\Sigma_k}\left\|x-z\right\|_{\lp}= \left\|x-x_{[k]}\right\|_{\lp}.\]
\end{definition}
\begin{remark}
The best $k$-term approximation error is the minimal distance of $x$ to a $k$-sparse vector. Informally, vectors having a relatively small best $k$-term approximation error are considered to be {\it compressible}. 
\end{remark}
\begin{remark}
If we define the \emph{nonincreasing rearrangement} of $x$ by
\[r(x) = (|x_{i_1}|,\ldots,|x_{i_N}|)^T,\;\mathrm{ and }\;|x_{i_j}|\geq|x_{i_{j+1}}|,\; j=1,\ldots,N-1,\]
then 
\[\sigma_k(x)_{\lp}= \left(\sum\limits_{j=k+1}^N r_j(x)^p\right)^{\frac{1}{p}}, \quad 1\leq p < \infty.\]
Alternatively, we can describe the best-$k$-term approximation error by
\[\sigma_k(x)_{\lp}= \left(\sum\limits_{j\in\Lambda^c}|x_j|^p\right)^{\frac{1}{p}},\]
where $\Lambda\ldef \supp(x_{[k]})$, and $\Lambda^c$ is its complement in $\{1,\dots,N\}$.
\end{remark}
A desirable property of an encoder/decoder pair $(A,\Delta)$ is given by the following stability estimate, called \textit{instance optimality}
\begin{equation}
\label{eq:optimaldecoder}
\left\|x-\Delta(Ax)\right\|_{\lp}\leq C \sigma_k(x)_{\lp},
\end{equation}
for all $x\in\R^N$, with a positive constant $C$ independent of $x$, and $k$ the closest possible to $m$ \cite{codade09}. This would in particular imply that by means of $\Delta$ we are able to recover a $k$-sparse signal $x$ exactly, since in this case $\sigma_k(x)_{\lp}=0$. It turns out that the existence of such a pair 
restricts the range of  $k$ to be maximally of the order of $\frac{m}{\log\frac{m}{N}+1}$. We refer to \cite{badadewa08,codade09,do06} for more details.

Actually, the above mentioned condition~\eqref{eq:optimaldecoder} can be realized in practice, at least for $p=1$, by pairing the \emph{$\lone$-minimization} as the decoder with the choice of an encoder which has the so-called \emph{Null Space Property} of optimal order $k$. (For realizations of the instance optimality in other $\ell_p$-norms, for instance for $p=2$, one needs more restrictive requirements, see \cite{wo08}. For the analysis within this paper, we shall use \eqref{eq:optimaldecoder} just for $p=1$ for the sake of simplicity.)
\begin{definition}[Null Space Property]
A matrix $A\in\R^{m\times N}$ has the \emph{Null Space Property of order $k$ and for positive constant $\gamma_k>0$} if 
\[\vectornorm{z|_{\Lambda}}_{\lone}\leq\gamma_k\vectornorm{z|_{\Lambda^c}}_{\lone},\] 
for all $z\in\ker A$ and all $\Lambda\subonetoN$ such that $\#\Lambda\leq k$. We abbreviate this property with the writing \emph{$(k,\gamma_k)$-NSP}.
\end{definition}
The Null Space Property states that the kernel of the encoding matrix $A$ contains no vectors where some entries have a significantly larger magnitude with respect to the others. In particular, no compressible vector is contained in the kernel. This is a natural requirement since otherwise no decoder would be able to distinguish a sparse vector from zero.

\begin{lemma}
\label{lem:l1performance}
Let $A\in\R^{m\times N}$ have the $(k,\gamma_k)$-NSP, with $\gamma_k <1$, and define 
$$
\F(y)\ldef\{z\in\R^N|Az=y\},$$ 
the set of feasible vectors for the measurement vector $y\in\R^m$. Then the decoder
\begin{equation}\label{prob:l1}
\Delta_1(y)\ldef\argmin\limits_{z\in\F(y)}\vectornorm{z}_{\lone},
\end{equation}
which we call \emph{$\lone$-minimization}, performs
\begin{equation}
\label{eq:optimaldecoderl1}
\vectornorm{x-\Delta_1(y)}_{\lone}\leq C \sigma_k(x)_{\lone},
\end{equation}
 for all $x\in\F(y)$ and the constant $C\ldef\frac{2(1+\gamma_k)}{1-\gamma_k}$.
\end{lemma}
Although the result we stated here is by now well-known, see also \cite{MXAH08}, we report its short proof for the sake of completeness, and for comparison with the enhanced guarantees given in Theorem \ref{thm:mainth} below.

\begin{proof} Let us denote $x^* = \Delta_1(y)$ and $z = x^* - x$. Then $z \in \ker A$ and
$$
\|x^*\|_{\ell_1} \leq \|x\|_{\ell_1},
$$
because $x^*$ is a solution of the $\ell_1$-minimization problem \eqref{prob:l1}.
Let $\Lambda$ be the set of the $k$-largest entries of $x$ in absolute value. One has
$$
\|x^*|_\Lambda \|_{\ell_1} + \|x^*|_{\Lambda^c} \|_{\ell_1} \leq \|x|_\Lambda \|_{\ell_1} + \|x|_{\Lambda^c} \|_{\ell_1}.
$$
It follows immediately from the triangle inequality that
$$
\|x|_\Lambda \|_{\ell_1} -\|z|_\Lambda \|_{\ell_1} +\|z|_{\Lambda^c} \|_{\ell_1}  - \|x|_{\Lambda^c} \|_{\ell_1} \leq \|x|_\Lambda \|_{\ell_1} + \|x|_{\Lambda^c} \|_{\ell_1}.
$$
Hence, by the $(k,\gamma_k)$-NSP
$$
\|z|_{\Lambda^c} \|_{\ell_1}\leq \|z|_\Lambda \|_{\ell_1} +2 \|x|_{\Lambda^c} \|_{\ell_1} \leq 
\gamma_k \| z|_{\Lambda^c}\|_{\ell_1} + 2 \sigma_k(x)_{\ell_1}, 
$$
or, equivalently,
\begin{equation}
\label{eq:aux}
\|z|_{\Lambda^c} \|_{\ell_1} \leq 
\frac{2}{1-\gamma_k} \sigma_k(x)_{\ell_1}.
\end{equation}
Finally, again by  the $(k,\gamma_k)$-NSP
\begin{eqnarray*}
\|x - x^*\|_{\ell_1} = \|z|_\Lambda \|_{\ell_1} + \|z|_{\Lambda^c} \|_{\ell_1}
\leq (\gamma_k+1) \|z|_{\Lambda^c} \|_{\ell_1}
\leq \frac{2(1+\gamma_k)}{1-\gamma_k} \sigma_k(x)_{\ell_1},
\end{eqnarray*}
and the proof is completed.
\end{proof}

Unfortunately, the NSP is hard to verify in practice. Therefore one can introduce another property which is called the \emph{Restricted Isometry Property} which implies the NSP.
Being a spectral concentration property, the Restricted Isometry Property is particularly suited to be verified with high probability by certain random matrices; we
mention some instances of such classes of matrices below.
\begin{definition}[Restricted Isometry Property]
\label{def:rip}
A matrix $A\in\R^{m\times N}$ has the \emph{Restricted Isometry Property (RIP) of order $K$ with constant $0<\delta_K<1$} if
\[(1-\delta_{K})\vectornorm{z}_{\ltwo} \leq \vectornorm{Az}_{\ltwo} \leq (1+ \delta_{K})\vectornorm{z}_{\ltwo},\] 
for all $z\in\Sigma_K$. We refer to this property by \emph{$(K,\delta_K)$-RIP}.
\end{definition}
\begin{lemma}
Let $k,h\in\N^+$ and $K = k+h$. Assume that $A\in\R^{m\times N}$ has $(K,\delta_K)$-RIP. Then $A$ has $(k,\gamma_k)$-NSP, where
\[\gamma_k\ldef\sqrt{\frac{k}{h}}\frac{1+\delta_K}{1-\delta_K}.\]
\end{lemma}
The proof of this latter result can be found, for instance, in \cite{fora10-1}, and not being of specific  relevance for this paper we do not include it here, although it will be important at some points below to connect  the RIP to a corresponding NSP.
Encoders which have the RIP with optimal constants, i.e., with $k$ in the order of $\frac{m}{\log\frac{m}{N}+1}$ exist, but, so far, as mentioned above, they can be realized exclusively by randomization. By now, classical examples of stochastic encoders are i.i.d. Gaussian matrices \cite{do06} or discrete Fourier matrices with randomly chosen rows \cite{carota06-1}. Further details and generalizations are provided in \cite{fora10,ra09-1}. In the rest of the paper we will use as prototypical cases mainly such stochastic encoders.\\

A recent ansatz to increase the reconstruction accuracy of sparse vectors is iteratively re-weighted $\ell_1$-minimization~\cite{candesIRWL1,needellIRWL1}. The idea of this approach is to iteratively solve the weighted $\lone$-minimization
$$
z^{n+1}=\argmin\limits_{z\in\F(y)}\sum\limits_{i=1}^N w_i^n |z_i|,
$$
while updating the weights according to  $w_i^n = \left(|z_i^n|+a\right)^{-1}$ for all $i=1,\dots,N$, for a suitably chosen stability parameter $a>0$. We denote this decoder by $\Delta_{1\text{rew}}$, and recall the following respective instance optimality result.
\begin{lemma}[{\cite[Theorem 3.2]{needellIRWL1}}]
\label{lem:rewl1performance}
Let $A\in \R^{m\times N}$ have the $(2k,\delta_{2k})$-RIP, with $\delta_{2k} < \sqrt{2} - 1$, and assume the smallest nonzero coordinate of $x_{[k]}$ in absolute value larger than the threshold
\begin{equation}
\label{eq:rewr}
\bar{r}\ldef 9.6 \frac{\sqrt{1+\delta_{2k}}}{1-(\sqrt{2}+1)\delta_{2k}}\left(\sigma_k(x)_{\ell_2} +\frac{\sigma_k(x)_{\ell_1}}{\sqrt{k}}\right).
\end{equation}
Then the decoder $\Delta_{1\text{rew}}$ performs
\begin{equation}
\label{eq:optimaldecoderrewl1}
\vectornorm{x-\Delta_{1\text{rew}}(y)}_{\ltwo}\leq 4.8 \frac{\sqrt{1+\delta_{2k}}}{1+(\sqrt{2}-1)\delta_{2k}}\left(\sigma_k(x)_{\ell_2} +\frac{\sigma_k(x)_{\ell_1}}{\sqrt{k}}\right).
\end{equation}
\end{lemma}

What we recalled up to this point of the theory of compressive sensing tells us that we are able to recover by (iteratively re-weighted) $\ell_1$-minimization compressible vectors within a certain accuracy, given by \eqref{eq:optimaldecoderl1} or \eqref{eq:optimaldecoderrewl1} respectively. If we re-interpret compressible vectors as sparse vectors which are corrupted by noise, we immediately see that the accuracy of the recovered solution is basically driven by the noise level affecting the vector. Nevertheless, neither inequalities~\eqref{eq:optimaldecoderl1} and \eqref{eq:optimaldecoderrewl1} tell us immediately if the recovered support of the $k$ largest entries of the decoded vector in absolute value is the same as the one of the original signal nor are we able to identify the large entries exceeding a given threshold in absolute value. Section~\ref{sec:supportidentification} addresses these issues in detail and investigates the limitations of $\lone$-minimization and iteratively re-weighted $\lone$-minimization. 
\\
Furthermore, we propose a new decoder, which is able to outperform $\lone$-minimization and iteratively re-weighted $\lone$-minimization in terms of  simultaneously identifying the exact position of the significant entries and reducing the noise level on the signal.

\section{Support identification}\label{sec:SI}
\label{sec:supportidentification}
We have seen in Lemma~\ref{lem:l1performance} that sparse vectors can be recovered exactly by the ${\ell_1}$--minimization decoder $\Delta_1$ if the matrix has the NSP. Moreover, a sparse signal which is disturbed by noise is recovered within a certain accuracy depending on the best $k$-term approximation error. In this section, we investigate in detail the noise level we can tolerate without loosing the ability of $\lone$-minimization to recover the support of the undisturbed sparse signal.\\ 

For later use, let us denote, for $1 \leq p \leq 2$ and $q$ such that $\frac{1}{p}+\frac{1}{q}=1$,
\begin{equation}
\label{eq:kappap}
\kappap\ldef\kappap(N,k)\ldef\begin{cases}1, & p=1, \\ \sqrt[q]{N-k}, & 1<p\leq 2.\end{cases}
\end{equation} 

\subsection{The $\lone$-minimization  result}
\label{sec:suppidl1}
The following simple proposition shows how we can recover the support of the original signal if we know the $\lone$-minimizer. It turns out that the large entries of the original signal in absolute value need to exceed a certain threshold, which depends on the noise level.
\begin{theorem}\label{thm:l1guar}
Let $x\in\R^N$ be a noisy signal with $k$ relevant entries and the noise level $\eta\in\R$, $\eta\geq 0$, i.e., for $\Lambda = \supp( x_{[k]})$,
\begin{equation}
\label{eq:noiselevel}                                                                                                                                                                                                                                                                                                                                                                                                                                                                               \sum\limits_{j\in\Lambda^c} |x_j|^p\leq\eta^p,
\end{equation}
for a \textit{fixed} $1\leq p \leq 2$. Consider further an encoder $A\in\R^{m\times N}$ which has the $(k,\gamma_k)$-NSP, with $\gamma_k <1$, the respective measurement vector $y = Ax\in\R^m$, and the $\lone$-minimizer
\[x^*\ldef\argmin\limits_{z\in\F(y)}\vectornorm{z}_{\lone}.\]
If the $i$-th component of the original signal $x$ is such that
\begin{equation}
\label{gapcond1}
|x_i|>\frac{2(1+\gamma_k)}{1-\gamma_k}\kappap\;\eta,
\end{equation}
then $i\in\supp(x^*)$.
\end{theorem}
\begin{proof}
We know by~\eqref{eq:optimaldecoderl1} that 
\begin{equation}
\label{eq:carot}
\vectornorm{x^* - x}_{\lone} \leq \frac{2(1+\gamma_k)}{1-\gamma_k} \sigma_k(x)_{\lone}.
\end{equation}
Thus, by H\"older's inequality and the assumption~\eqref{eq:noiselevel}, we obtain the estimate
\begin{equation}
\label{eq:potato}
\vectornorm{x^* - x}_{\lone} \leq\frac{2(1+\gamma_k)}{1-\gamma_k}\sigma_k(x)_{\lone} \leq \frac{2(1+\gamma_k)}{1-\gamma_k}\kappap\;\eta.
\end{equation}
We now choose a component $i\in\{\onetoN\}$ such that 
\[|x_i|>\frac{2(1+\gamma_k)}{1-\gamma_k}\kappap\;\eta,\]
and assume $i\notin \supp(x^*)$. This leads to the contradiction:
\begin{equation}
|x_i| = |x_i - x_i^*| \leq \vectornorm{x -x^*}_{\lone} \leq \frac{2(1+\gamma_k)}{1-\gamma_k} \kappap\; \eta < |x_i|.
\end{equation}
Hence, necessarily $i\in \supp(x^*)$.
\end{proof}


The noise level substantially influences the ability of support identification. Here, the noisy signal should have (as a sufficient condition) the $k$ largest entries in absolute value above 
$$
r_1\ldef \frac{2(1+\gamma_k)}{1-\gamma_{k}} \kappap\; \eta,
$$
in order to guarantee support identification.

\subsection{The result for iteratively re-weighted $\ell_1$-minimization}
We are able also in the case of the iteratively re-weighted $\ell_1$-minimization to show a similar support identification result, which follows similarly from Lemma~\ref{lem:rewl1performance}. 
\begin{theorem}\label{thm:rewl1guar}
Let $x\in\R^N$ be a noisy signal with $k$ relevant entries and the noise level $\eta\in\R$, $\eta\geq 0$, i.e., for $\Lambda = \supp(x_{[k]})$,
\begin{equation}
\label{eq:rewnoiselevel}                                                                                                                                                                                                                                                                                                                                                                                                                                                                               \sum\limits_{j\in\Lambda^c} |x_j|^p\leq\eta^p,
\end{equation}
for a \textit{fixed} $1\leq p \leq 2$. Consider further an encoder $A\in\R^{m\times N}$ which has the $(2k,\delta_{2k})$-RIP, with $\delta_{2k} <\sqrt{2}-1$, the respective measurement vector $y = Ax\in\R^m$, and the iteratively re-weighted $\lone$-minimizer $x^*\ldef\Delta_{1\text{rew}}(y)$.
If for all $i\in \supp(x_{[k]})$ 
\begin{equation}
\label{gapcond2}
|x_i|>9.6\frac{\sqrt{1+\delta_{2k}}}{1-(\sqrt{2}+1)\delta_{2k}}\left(1+\frac{\kappap}{\sqrt{k}}\right)\;\eta,
\end{equation}
then $\supp(x_{[k]})\subset\supp(x^*)$.
\end{theorem}
\begin{proof}
First, notice that $\eta \geq \sqrt[p]{\sum\limits_{j\in\Lambda^c} |x_j|^p} = \sigma_k(x)_{\ell_p} \geq \sigma_k(x)_{\ell_2}$ and $ \kappap\sigma_k(x)_{\ell_p} \geq \sigma_k(x)_{\ell_1}$ by H\"older's inequality. Thus, we have for all $i\in \supp(x_{[k]})$ that
\begin{align*}
|x_i| &>9.6\frac{\sqrt{1+\delta_{2k}}}{1-(\sqrt{2}+1)\delta_{2k}}\left(1+\frac{\kappap}{\sqrt{k}}\right)\;\eta \geq
9.6\frac{\sqrt{1+\delta_{2k}}}{1-(\sqrt{2}+1)\delta_{2k}}\left( \sigma_k(x)_{\ell_p} +\frac{\kappap}{\sqrt{k}} \sigma_k(x)_{\ell_p} \right) \\
&\geq 9.6\frac{\sqrt{1+\delta_{2k}}}{1-(\sqrt{2}+1)\delta_{2k}}\left( \sigma_k(x)_{\ell_2} +\frac{\sigma_k(x)_{\ell_1}}{\sqrt{k} }\right) .
\end{align*}
Consequently, we fulfill the conditions of Lemma~\ref{lem:rewl1performance} for which, for all $i\in \supp(x_{[k]})$, $|x_i| > \bar{r}$, as defined in~\eqref{eq:rewr}. \\

Assume now that there is $i\in \supp(x_{[k]})$ and $i\notin \supp(x^*)$. By means of Lemma~\ref{lem:rewl1performance}, we obtain the contradiction
\begin{align}
|x_i| &= |x_i-x_i^*| \leq \vectornorm{x-x^*}_{\ltwo}\leq 4.8 \frac{\sqrt{1+\delta_{2k}}}{1+(\sqrt{2}-1)\delta_{2k}}\left(\sigma_k(x)_{\ell_2} +\frac{\sigma_k(x)_{\ell_1}}{\sqrt{k}}\right) \nonumber \\
 &\leq 9.6\frac{\sqrt{1+\delta_{2k}}}{1-(\sqrt{2}+1)\delta_{2k}}\left( \sigma_k(x)_{\ell_2} +\frac{\sigma_k(x)_{\ell_1}}{\sqrt{k} }\right) = \bar{r} < |x_i|. \label{rwl1suppid}
\end{align}
Hence,  $i\in \supp(x^*)$.
\end{proof} 

As negative aspects $\Delta_{1\text{rew}}$, we notice that, in view of \eqref{rwl1suppid} already the conditions of Lemma~\ref{lem:rewl1performance} for the stable convergence of the algorithm do imply support identification. 
Moreover, we stress that Theorem \ref{thm:rewl1guar} requires the RIP, which is in general a stronger condition than the NSP. Again, one observes that the noise level influences the ability of support identification. Here, the noisy signal should have the $k$ largest entries in absolute value above 
$$
r_{1\text{rew}}\ldef 9.6\frac{\sqrt{1+\delta_{2k}}}{1-(\sqrt{2}+1)\delta_{2k}}\left(1+\frac{\kappap}{\sqrt{k}}\right)\;\eta,
$$
in order to guarantee support identification.\\

In the following, we shall show that for the class $\Ss_{\eta,k,r}^p$, as defined in~\eqref{eq:Setakr}, a smaller threshold is required for support identification, under the sole request of the NSP (and not of the RIP).

\subsection{Support identification stability in the class $\Ss_{\eta,k,r}^p$}
\label{sec:signalclass}

In this section, we present results in terms of support discrepancy once we consider two elements of the class $\Ss_{\eta,k,r}^p$ defined in \eqref{eq:Setakr}, having the same measurements, i.e., they are both in $\F(y)$ for $y \in \mathbb R^m$. 
\begin{theorem}
\label{thm:suppid}
Let $A \in\R^{m\times N}$ have the $(2k,\gamma_{2k})$-NSP, for $\gamma_{2k} <1$, $1\leq p \leq 2$, and $x,x'\in\Ss_{\eta,k,r}^p$ such that $Ax=Ax'$, and  $0\leq \eta<r$. 
Then
\begin{equation}\label{eq:stabsupp}
\#(S_r(x)\Delta S_r(x'))\leq\frac{(2\gamma_{2k}\kappap\eta)^p}{(r-\eta)^p}.
\end{equation}
(Here we denote by ``$\Delta$'' the set symmetric difference, not to be confused with the symbol of a generic decoder.)
If additionally
\begin{equation}\label{eq:thrs2}
r> \eta(1+2\gamma_{2k}\kappap)\rdef r_S,
\end{equation}
then $S_r(x) = S_r(x')$, i.e., we have unique identification of the large entries in absolute value.
\end{theorem}
\begin{proof}
As $Ax=Ax'$, the difference $(x-x')\in \ker(A)$. By the $(2k,\gamma_{2k})$-NSP, H\"older's inequality, and the triangle inequality we have
\begin{align}
\nonumber\vectornorm{(x-x')|_{S_r(x)\cup S_r(x')}}_{\lp}&\leq\vectornorm{(x-x')|_{S_r(x)\cup S_r(x')}}_{\lone}\\
\nonumber&\leq\gamma_{2k}\vectornorm{(x-x')|_{(S_r(x)\cup S_r(x'))^c}}_{\lone} \\
\nonumber&\leq\gamma_{2k}\kappap\vectornorm{(x-x')|_{(S_r(x)\cup S_r(x'))^c}}_{\lp}\\
\label{eq:estim1} & \leq 2\gamma_{2k}\kappap\eta.
\end{align}
Now we estimate the symmetric difference of the supports of the large entries of $x$ and $x'$ in absolute value as follows: if $i\in S_r(x) \Delta S_r(x')$, then either $|x_i| > r$ and $|x'_i|\leq\eta$ or $|x_i| \leq\eta$ and $|x'_i|>r$. This implies that $|x'_i-x_i|>(r-\eta)$. Thus we have
\[\vectornorm{(x-x')|_{S_r(x)\Delta S_r(x')}}_{\lp}^p \geq \left(\#(S_r(x)\Delta S_r(x'))\right)(r-\eta)^p.\]
Together with the inequality~\eqref{eq:estim1}  and the non-negativity of $\vectornorm{(x-x')|_{S_r(x)\cap S_r(x')}}_{\lp}$, we obtain the chain of inequalities 
\begin{align*}
(2\gamma_{2k}\kappap\eta)^p&\geq \vectornorm{(x-x')|_{S_r(x)\cup S_r(x')}}_{\lp}^p \\
&\geq\vectornorm{(x-x')|_{S_r(x)\cap S_r(x')}}_{\lp}^p + \vectornorm{(x-x')|_{S_r(x)\Delta S_r(x')}}_{\lp}^p \\
&\geq\left(\#(S_r(x)\Delta S_r(x'))\right)(r-\eta)^p,
\end{align*}
and therefore
\begin{equation}
\label{eq:diffSr}
\#(S_r(x)\Delta S_r(x'))\leq\frac{(2\gamma_{2k}\kappap\eta)^p}{(r-\eta)^p}.
\end{equation}
For the unique support identification, we want the symmetric difference between the sets $S_r(x)$ and $S_r(x')$ to be empty. Thus the left-hand side of inequality~\eqref{eq:diffSr} has to be zero. Since $\#(S_r(x)\Delta S_r(x'))\in\N$, it is sufficient to require that the right-hand side be strictly less than one, and this is equivalent to condition \eqref{eq:thrs2}. 
\end{proof}

\begin{remark}
\label{rem:precBG}
One additional implication of this latter theorem is that we can give a bound on the difference of $x$ and $x'$ restricted to the relevant entries. Indeed, in case of unique identification of the relevant entries, i.e., $\Lambda \ldef S_r(x) = S_r(x')$ we obtain, by the inequality ~\eqref{eq:estim1}, that
\begin{equation}
\label{eq:estimBigGuys}
\vectornorm{(x-x')_{\Lambda}}_{\lone}\leq 2\gamma_k \kappa_p\eta.
\end{equation}
Notice that we replaced $\gamma_{2k}$ by $\gamma_k \leq \gamma_{2k}$, because now $\# \Lambda \leq k$.
\end{remark}
\begin{remark}
Unfortunately, we are not able to provide the {\it necessity of the gap conditions} \eqref{gapcond1}, \eqref{gapcond2}, \eqref{eq:thrs2} for successful support recovery, simply because we lack optimal deterministic error bounds in general: one way of producing a lower bound would be to construct a counterexample for each algorithm for which
a certain gap condition is violated and recovery of support fails. Since most of the algorithms we shall illustrate below are iterative, it
is likely extremely difficult to provide such explicit counterexamples. Therefore, we limit ourselves here to discuss the difference of $r_1$ and $r_S$ and of $r_{1\text{rew}}$ and $r_S$. We shall see in the numerical experiments that
the {\it sufficient gap conditions} \eqref{gapcond1}, \eqref{gapcond2}, \eqref{eq:thrs2} nevertheless provide actual indications of performance of the algorithms.
\begin{itemize}
\item The gap between the two thresholds $r_1,r_S$ is given by
$$
r_1 - r_S = \left ( 2 \left ( \frac{1+ \gamma_k}{1-\gamma_k} - \gamma_{2 k} \right) \kappa_p(N,k) - 1 \right) \eta.
$$
As $\gamma_{2k } < 1 < \frac{1+ \gamma_k}{1-\gamma_k}$ and $\kappa_p(N,k)$ is very large for $N \gg k$, this \textit{positive} gap is actually very large, for $N \gg 1$.
\item The gap between the two thresholds $r_{1\text{rew}},r_S$ is given by
$$
r_{1\text{rew}} - r_S = \left(9.6\frac{\sqrt{1+\delta_{2k}}}{1-(\sqrt{2}+1)\delta_{2k}}\left(1+\frac{\kappap}{\sqrt{k}}\right) - (1+2\gamma_{2k}\kappap)\right) \eta.
$$

Following for example the arguments in \cite{davenportRIP}, we know that the matrix $A$ having the $(2k,\delta_{2k})$-RIP implies having the $(2k,\gamma_{2k})$-NSP with $\gamma_{2k}=\frac{\sqrt{2}\delta_{2k}}{1-(\sqrt{2}+1)\delta_{2k}}$, which, substituted into the above equation, yields
$$
r_{1\text{rew}} - r_S = \left(\frac{\left(\frac{9.6\sqrt{1+\delta_{2k}}}{\sqrt{k}}-2\sqrt{2}\delta_{2k}\right)\kappap + \left[9.6\sqrt{1+\delta_{2k}} - (1-(\sqrt{2}+1)\delta_{2k})\right]}{1-(\sqrt{2}+1)\delta_{2k}}\right) \eta.
$$
Since $0 < \delta_{2k} < \sqrt{2} -1$, we have $0 <  1-(\sqrt{2}+1)\delta_{2k} < 1$, and therefore the denominator and the right summand in the numerator are positive. The left summand of the numerator is positive and very large as soon as $k < \left(\frac{9.6\sqrt{1+\delta_{2k}}}{2\sqrt{2}\delta_{2k}}\right)^2$. Thus, even in the limiting scenario where $\delta_{2k} \approx \sqrt{2} - 1$, we still have $ k \leq 94$, which may be considered sufficient for a wide range of applications. If one wants to exceed this threshold a more sophisticated estimate of the above term will reveal even less restrictive bounds on $k$. Thus, in general, since $k$ and $\delta_{2k}$ are small, also the left summand is positive. We conclude that again the gap is very large.
\end{itemize}
Unfortunately, this discrepancy cannot be amended because the $\ell_1$-minimization decoder $\Delta_1$ and the iteratively re-weighted $\ell_1$-minimization decoder $\Delta_{1\text{rew}}$ have \textit{not} in general the property 
$$
x \in \Ss_{\eta,k,r}^p \Rightarrow \{ \Delta_1(Ax), \Delta_{1\text{rew}}(Ax) \} \ni x^* \in \Ss_{\eta,k,r}^p.
$$
Hence, it allows us neither to say that also the (iteratively re-weighted) $\ell_1$-minimizer has a bounded noise component
$$
\sum_{i\in(S_r(x^*))^c}|x_i^*|^p\leq\eta^p,
$$ 
nor to apply Theorem \ref{thm:suppid} to obtain support stability. We present several 
examples in Section \ref{sec:NR}, showing these ineliminable limitations of $\Delta_1$ and $\Delta_{1\text{rew}}$.
\end{remark}



\subsection{The regularized selective $p$-potential functional and its properties}

To overcome the shortcomings of methods based exclusively on $\ell_1$-minimizations in 1.~damping the noise-folding and consequently in 2.~having 
a stable support recovery, in this section, we design a new decoding procedure which allows us to have both these very desirable properties.  

Let us first introduce the following functional. 

\begin{definition}[Regularized selective $p$-potential]
We define the \emph{regularized truncated $p$-power function} $W_r^{p,\epsilon}\colon \R \rightarrow \R_0^+$ by
\begin{equation}\label{eq:truncqp}
W_r^{p,\epsilon}(t)=\left\{\begin{array}{ll}
                                          t^p & 0\leq t<r-\epsilon,\\
                                          \pi_p(t) & r-\epsilon\leq t \leq r+\epsilon,\\
                                          r^p & t>r+\epsilon,
                                       \end{array}\right. \quad t \geq 0,
\end{equation}
where $0<\varepsilon <r$, and $\pi_p(t)$ is the third degree  interpolating polynomial
$$
\pi_p(t) := A(t-s_2)^3 +B(t-s_2)^2 + C,
$$
with
\begin{equation}\label{Bdefin}
\left \{ 
\begin{array}{l}
C=\mu_3, \\
B=\frac{\mu_1}{s_2-s_1} - \frac{3 (\mu_3 - \mu_2)}{(s_2 -s_1)^2},\\
A=\frac{\mu_1}{3(s_2-s_1)^2} + \frac{2 B}{3(s_2 -s_1)}.
\end{array}
\right .
\end{equation}
The parameters which appeared in the definition of the interpolating function are defined as: $s_1 = (r-\varepsilon)$, $s_2=(r+\varepsilon)$, $\mu_1 = p(r-\varepsilon)^{p-1}$,
$\mu_2 = (r-\varepsilon)^p$, and $\mu_3 = r^p$. Moreover, we define it for $t<0$ as $W_r^{p,\epsilon}(t)=W_r^{p,\epsilon}(-t)$.
We call the functional $\SLP_r^{p,\epsilon} \colon \R^N \rightarrow \R_0^+$,
\begin{equation} \label{eq:decodeps}
\SLP_r^{p,\epsilon}(x)=\sum\limits_{j=1}^N W_r^{p,\epsilon}(x_j),\qquad r>0,\qquad 1 \leq p \leq 2,
\end{equation}
the \emph{regularized selective $p$-potential (SP) functional}.
\end{definition}
To easily capture the function $W_r^{p,\epsilon}$, we express it explicitly for $p=2$. In this case, the polynomial has the analytic form
$$
\pi_2(t):=\frac{[t+(r-\epsilon)][\epsilon(r+t)-(r-t)^2]}{4\epsilon},
$$
and the graph of $W_r^{p,\epsilon}$ is shown in Figure~\ref{trunc} for $p=2$, $r = 1$, and $\epsilon = 0.4$. Notice that $W_r^{p,0}$ is the truncated p-power function, which is widely used both in statistics and signal processing \cite{A.SolForn12,fowa10}, and also shown in Figure~\ref{trunc} for $p=2$.
\begin{figure}[htp]
\begin{center} 
\includegraphics[width=4.2in]{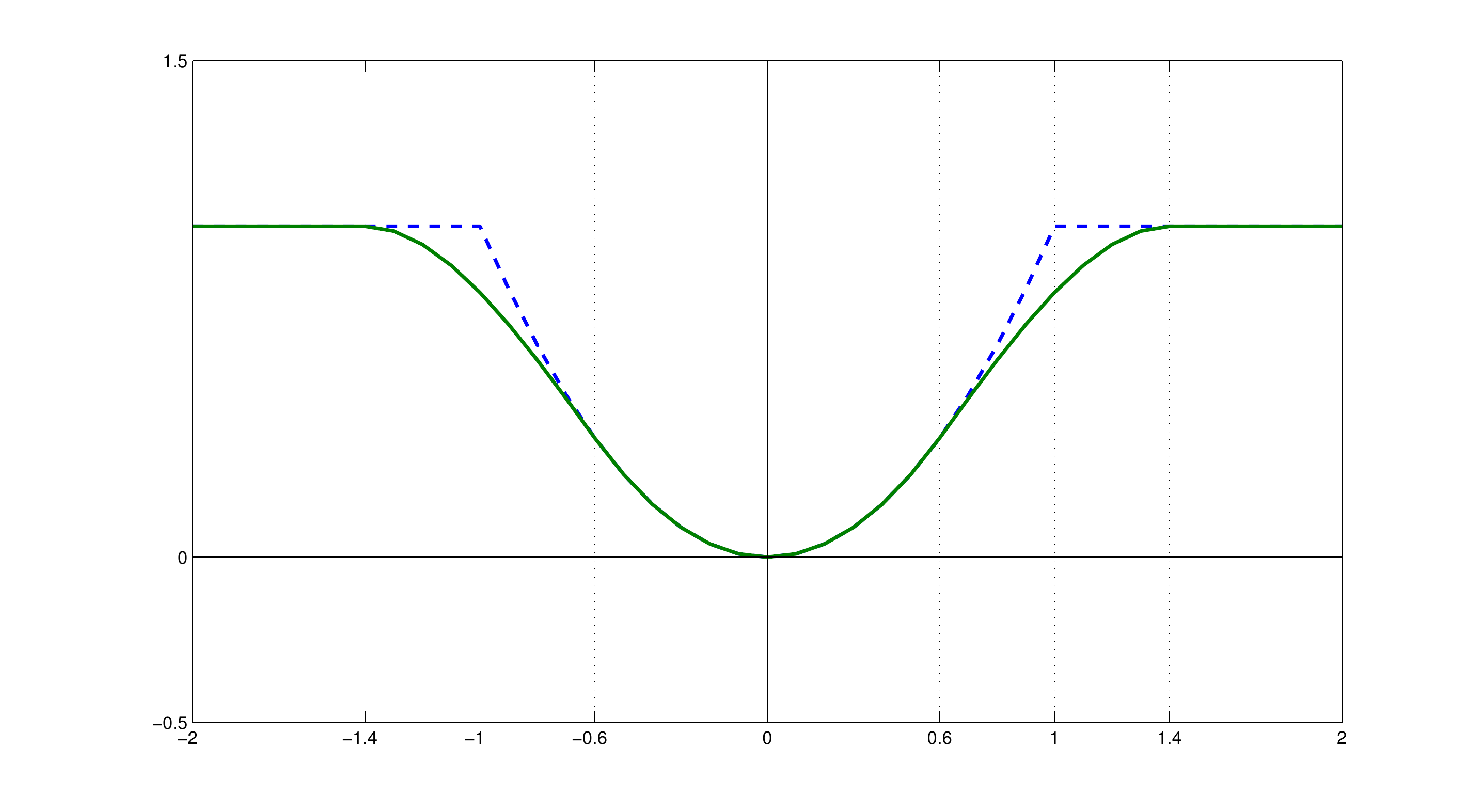}
\end{center}
\caption{Truncated quadratic (2-power) potential $W_1^{2,0}$ (dashed) and its regularization $W_1^{2,0.4}$.}\label{trunc}
\end{figure}

\begin{remark}
Let us notice that the functional $ {\SLP}_r^{p,\epsilon}$ is semi-convex for $\epsilon > 0$, that means that there exists a constant $\omega>0$ such that ${\SLP}_r^{p,\epsilon}(\cdot)+\omega\vectornorm{\cdot}^2$ is convex.
\end{remark}

\begin{theorem}\label{thm:mainth}
Let $A \in\R^{m\times N}$ have the $(2k,\gamma_{2k})-NSP$, with $\gamma_{2k} <1$, and $1 \leq p \leq 2$. Furthermore, we assume $x\in\Ss_{\eta,k,r+\epsilon}^p$, for $\epsilon>0$, $0<\eta < r+\epsilon$, with the property of having the minimal $\#S_{r+\epsilon}(x)$ within $\F(y)$, where $y= Ax$  is its associated measurement vector, i.e., 
\begin{equation}\label{eq:minsupp} 
\#S_{r+\epsilon}(x) \leq \#S_{r+\epsilon}(z) \text{ for all } z \in  \F(y) .
\end{equation} 
If $x^*$ is such that 
\begin{equation}
\label{eq:minMS}
\SLP_r^{p,\epsilon}(x^*)\leq \SLP_r^{p,\epsilon}(x),
\end{equation}
and 
\begin{equation}
\label{eq:assumptionxstar}
|x_i^*| < r-\epsilon,
\end{equation}
for all $i \in \left(S_{r+\epsilon}(x^*)\right)^c$,
then also $x^*\in\Ss_{\eta,k,r+\epsilon}^p$, implying noise-folding damping. Moreover, we have the support stability property 
\begin{equation}\label{eq:suppstabprop}\#(S_{r+\epsilon}(x)\Delta S_{r+\epsilon}(x^*))\leq\frac{(2\gamma_{2k}\kappap\eta)^p}{(r+\epsilon-\eta)^p}.\end{equation}
\label{thm:MSsupp}
\end{theorem}
\begin{proof}
Notice that we can equally rewrite the $\SLP_r^{p,\epsilon}$ functional as
\[\SLP_r^{p,\epsilon}(z) = r^p\#S_{r+\epsilon}(z)+\sum\limits_{i\in\left(S_{r+\epsilon}(z)\right)^c}|z_i|_{\epsilon}^p,\]
where $|t|_{\epsilon}^p:=W_r^{p,\epsilon}(t)$ for $|t|\leq r+\epsilon$. Here, by construction, we have $|t|_{\epsilon}^p\leq |t|^p$.
By the assumptions~\eqref{eq:minMS} and $x \in \Ss_{\eta,k,r+\epsilon}^p$, we have
\begin{align}
\nonumber r^p\#S_{r+\epsilon}(x^*) &\leq\SLP_r^{p,\epsilon}(x^*) \leq\SLP_r^{p,\epsilon}(x) = r^p\#S_{r+\epsilon}(x)+\sum\limits_{i\in\left(S_{r+\epsilon}(x)\right)^c}|x_i|_{\epsilon}^p\\
& \leq r^p\#S_{r+\epsilon}(x)+\sum\limits_{i\in\left(S_{r+\epsilon}(x)\right)^c}|x_i|^p\leq r^p\#S_{r+\epsilon}(x)+ \eta^p,
\end{align}
and thus,
\[\#S_{r+\epsilon}(x^*)\leq\left(\frac{\eta}{r}\right)^p+\#S_{r+\epsilon}(x).\]
As  $\frac{\eta}{r}<1$ by assumption, the minimality property \eqref{eq:minsupp} yields immediately 
\begin{equation}
\label{eq:equalSr}
\#S_{r+\epsilon}(x^*)=\#S_{r+\epsilon}(x) \leq k.
\end{equation}
Assumption~\eqref{eq:assumptionxstar} and again~\eqref{eq:minMS} yield
\begin{align*}
r^p \#S_{r+\epsilon}(x^*)+\sum\limits_{i\in\left(S_{r+\epsilon}(x^*)\right)^c}|x_i^*|^p & = r^p \#S_{r+\epsilon}(x^*)+\sum\limits_{i\in\left(S_{r+\epsilon}(x^*)\right)^c}|x_i^*|_{\epsilon}^p \\
&\leq r^p\#S_{r+\epsilon}(x)+\sum\limits_{i\in\left(S_{r+\epsilon}(x)\right)^c}|x_i|_{\epsilon}^p\\
&\leq r^p\#S_{r+\epsilon}(x)+\sum\limits_{i\in\left(S_{r+\epsilon}(x)\right)^c}|x_i|^p.
\end{align*}
By this latter inequality and~\eqref{eq:equalSr} we obtain
\begin{equation*}
\sum\limits_{i\in\left(S_{r+\epsilon}(x^*)\right)^c}|x_i^*|^p \leq \sum\limits_{i\in\left(S_{r+\epsilon}(x)\right)^c}|x_i|^p \leq \eta^p,
\end{equation*}
which implies $x^*\in \Ss_{\eta,k,r+\epsilon}^p$. By an application of Theorem~\ref{thm:suppid}, we obtain~\eqref{eq:suppstabprop}.
\end{proof}

\begin{remark}\label{remarkPP}
Let us comment on the assumptions of the latter result.
\begin{enumerate}[(i)]
\item The assumption that $x$ is actually the vector with minimal essential support $S_r(x)$ among the feasible
vectors in $\F(y)$ corresponds to the request of being the "simplest" explanation to the data, in a certain sense we
acknowledge the Occam's razor principle;
\item The best candidate $x^*$ to fulfill condition~\eqref{eq:minMS} would be actually
\begin{equation}
\label{eq:globMin}
x^*:=\argmin_{z \in \F(y)}\SLP_r^{p,\epsilon}(z)
\end{equation}
because this will make~\eqref{eq:minMS} true, whichever $x$ is. However~\eqref{eq:globMin} is a highly nonconvex problem whose solution is in general NP-hard~\cite{alwa10}.
The way we will circumvent this drawback is to employ an algorithm, which we describe in details in the following section, to compute $x^*$ by âperforming a local minimization of $\SLP_r^{p,\epsilon}$ in $\F(y)$ around a given starting vector $x_0$. Ideally, the  best choice for $x_0$ would be $x$ itself, so that~\eqref{eq:minMS} may be fulfilled. But, obviously, we do not dispose yet of the original vector $x$. Therefore, a heuristic rule, which we will show to be very robust in our numerical simulations, is to choose $x_0=\Delta_1(y) \approx x$, i.e., we use  $x_0$ as the result of the $\ell_1$-minimization as a warm-up for the iterative algorithm described below. The reasonable hope is that actually 
$$
\SLP_r^{p,\epsilon}(x^*) \leq \SLP_r^{p,\epsilon}(\Delta_1(y)) \approx \SLP_r^{r,\epsilon}(x);
$$
\item The assumption that the outcome $x^*$ of the algorithm has additionally the property $|x_i^*| < r-\epsilon$,
for all $i \in \left(S_{r+\epsilon}(x^*)\right)^c$ is justified by observing that in our implementation $x^*$ will be the result 
of a thresholding operation, i.e., $x_i^* = \mathbb S_p^\mu(\xi_i)$, for $i \in \left(S_{r+\epsilon}(x^*)\right)^c$, see \eqref{thrsal}. The particularly steep shape of the thresholding function in the interval $[r-\epsilon,r+\epsilon]$,
especially for $p=2$, see Figure~\ref{s1(c)}, makes it highly unlikely  for $\epsilon$ sufficiently small that $r-\epsilon \leq |x_i^*|$ for $i \in \left(S_{r+\epsilon}(x^*)\right)^c$. Actually, our numerical experiments confirm that typically this algorithm promotes solutions in $\F(y)$ selectively characterized by a relevant gap between large components exceeding $r$ and small components, significantly below $r$.
\end{enumerate}
\end{remark}


\section{Minimization of the regularized selective $p$-potential functional}\label{sec:MSmin}
\label{sec:MSmin}
In the latter section, we introduced the functional $\SLP_r^{p,\epsilon}$, which is nonconvex. Unluckily, this makes its linearly constrained minimization~\eqref{eq:globMin}, which we want to call selective least $p$-powers (SLP), also nontrivial. Here, we recall a novel and very robust algorithm for linearly constrained nonconvex and nonsmooth minimization, introduced and analyzed first in \cite{A.SolForn12}. The algorithm is particularly suited for our purpose, since it only requires a $C^1$-regular functional. This distinguishes it from other well-known methods such as SQP and Newton methods, which require a more restrictive $C^2$-regularity. All notions and results written in this section are collected more in general in \cite{A.SolForn12}. Nevertheless we report them directly adapted to our specific case in order to have a simplified and more immediate application.\\

\subsection{The algorithm}
In this section, we present the algorithm to perform the local minimization of \eqref{eq:decodeps} in $\F(y)$.
Before describing it, it is necessary to introduce the concept of $\nu$-convexity, which plays a key-role in the minimization process. In fact, to achieve the minimization of the functional ${\SLP}_r^{p,\epsilon}$, we use a Bregman-like inner loop, which requires this property to converge with an a priori rate. 
\begin{definition}[$\nu$-convexity]
A function $f:\mathbb R^N \rightarrow \mathbb R$ is $\nu$-convex if there exists a constant $\nu>0$ such that for all $x, x'\in \mathbb R^N$ and $\phi\in\partial f(x),\psi\in\partial f(x')$
\begin{equation}
\langle\phi-\psi,x-x'\rangle\geq\nu\vectornorm{x-x'}_{\ltwo}^2,
\end{equation}
where $\partial f$ is the subdifferential of the function $f$. 
\end{definition}
The starting values $x_0=x_{(0,0)}\in\R^N$ and $q_{(0,0)}\in\R^m$ are taken arbitrarily. For a fixed scaling parameter $\lambda>0$ and an adaptively chosen sequence of integers $(L_\ell)_{\ell\in\N}$, we formulate 

\begin{algorithm}
\caption{SLP}
\label{ncbreg}
\begin{algorithmic}
\While{$\vectornorm{x_{\ell-1}-x_\ell}_{\ltwo}\leq TOL$}
 \State{$x_{(\ell,0)}=x_{\ell-1}:=x_{(\ell-1,L_{\ell-1})}$}
 \State{$ q_{(\ell,0)}=q_{\ell-1}:=q_{(\ell-1,L_{\ell-1})}$}
 \For{$k=1,\dots,L_\ell$}
 \State{$x_{(\ell,k)} = \argmin_{x \in \R^N} \,\big( {\SLP}^{p,\epsilon}_{\omega, x_{\ell-1}}(x)-\langle q_{(\ell,k-1)}, Ax \rangle + \lambda \vectornorm{Ax-y}^2_{\ltwo}\big)$}
 \State{$q_{(\ell,k)} = q_{(\ell,k-1)}+ 2\lambda (y-Ax_{(\ell,k)})\,$}
 \EndFor
\EndWhile
\end{algorithmic}
\end{algorithm} 

The reader can notice that the functional ${\SLP}^{p,\epsilon}_{\omega, x_{\ell-1}}$, which appears in the algorithm, has not been yet introduced. Indeed, a modification to the functional ${\SLP}^{p,\varepsilon}_r$ must be introduced in order to have $\nu-$convexity, which is necessary for the convergence of the algorithm. It is defined as
$$
{\SLP}^{p,\epsilon}_{\omega, x'}(x)\ldef {\SLP}^{p,\epsilon}_r(x)+\omega\vectornorm{x-x'}^2_{\ltwo},
$$
where $\omega$ is chosen such that the new functional is $\nu-$convex. The finite adaptively chosen number of inner loop iterates $L_\ell$ is defined by the condition
$$
(1+\vectornorm{q_{\ell-1}}_{\ltwo})\vectornorm{Ax_{(\ell,L_\ell)}-y}_{\ltwo}\leq\frac{1}{\ell^{\alpha}},
$$
for a given parameter $\alpha>1$, which in our numerical experiments will be set to $\alpha = 1.1$. We refer to \cite[Section 2.2]{A.SolForn12} for details on the finiteness of $L_\ell$ and for the proof of convergence of Algorithm \ref{ncbreg} to critical points
of $\SLP^{p,\epsilon}_r$ in $\F(y)$. According to Remark \ref{remarkPP} (ii), and as we will empirically verify in our numerical experiments reported in Section~\ref{sec:NR}, this algorithm finds critical points (hopefully close to a global minimizer) with the desired properties illustrated in Theorem \ref{thm:MSsupp},
as soon as we select the starting point $x_0$ by an appropriate warm-up procedure.\\

Algorithm \ref{ncbreg} does not yet specify how to minimize the convex functional
$$
\big( {\SLP}^{p,\epsilon}_{\omega, x_{\ell-1}}(x)-\langle q_{(\ell,k-1)}, Ax \rangle + \lambda \vectornorm{Ax-y}^2_{\ltwo}\big),
$$ in the inner loop. For that we can use an iterative thresholding algorithm introduced in \cite[Section 3.7]{A.SolForn12}, inspired by the previous work \cite{fowa10} for the corresponding \textit{unconstrained} optimization of regularized selective $p$-potentials. This method ensures the convergence to a minimizer and is extremely agile to be implemented, as it is based on matrix-vector multiplications and
very simple componentwise nonlinear thresholdings.\\

By the iterative thresholding algorithm, we actually equivalently minimize the functional
$$
{\SLP}^{p,\epsilon}_{\omega,x'}(x,q)={\SLP}^{p,\epsilon}_{\omega,x'}(x)+\lambda\vectornorm{Ax-(y+q)}_{\ltwo}^2,
$$
where we set $\lambda= \frac{1}{2}$ only for simplicity. The thresholding functions $\mathbb S_p^\mu$ we use are defined in \cite[Lemma 3.15]{A.SolForn12} and, for the relevant case $p=2$, it has the analytic
form
\begin{equation}
\mathbb S_2^\mu(\xi):=\left\{\begin{array}{ll}
                                   \displaystyle\frac{\xi}{1+\mu}&|\xi|<(\br-\epsilon)(1+\mu),\\ \ \\
                                   \displaystyle\frac{4\epsilon}{3\mu}\left(1+\frac{\mu}{4\epsilon}(2\epsilon+\br)-\sqrt{\frac{\Gamma(\xi)}{4}}\right) & (\br-\epsilon)(1+\mu)\leq|\xi|\leq r+\epsilon,\\ \ \\
                                   \xi & |\xi|>\br+\epsilon,\\
                      \end{array}\right.
\end{equation}
where
$$
\Gamma(\xi)\ldef4\left(1+\left(\frac{\mu}{4\epsilon}\right)^2(2\br+\epsilon)^2+\frac{\mu}{2\epsilon}(2\epsilon+\br)-\frac{3\mu}{2\epsilon}\xi\right).
$$

\begin{figure}[!ht]
\begin{center}
\subfigure[]{ \label{s1(a)}
\includegraphics[width=2.3in]{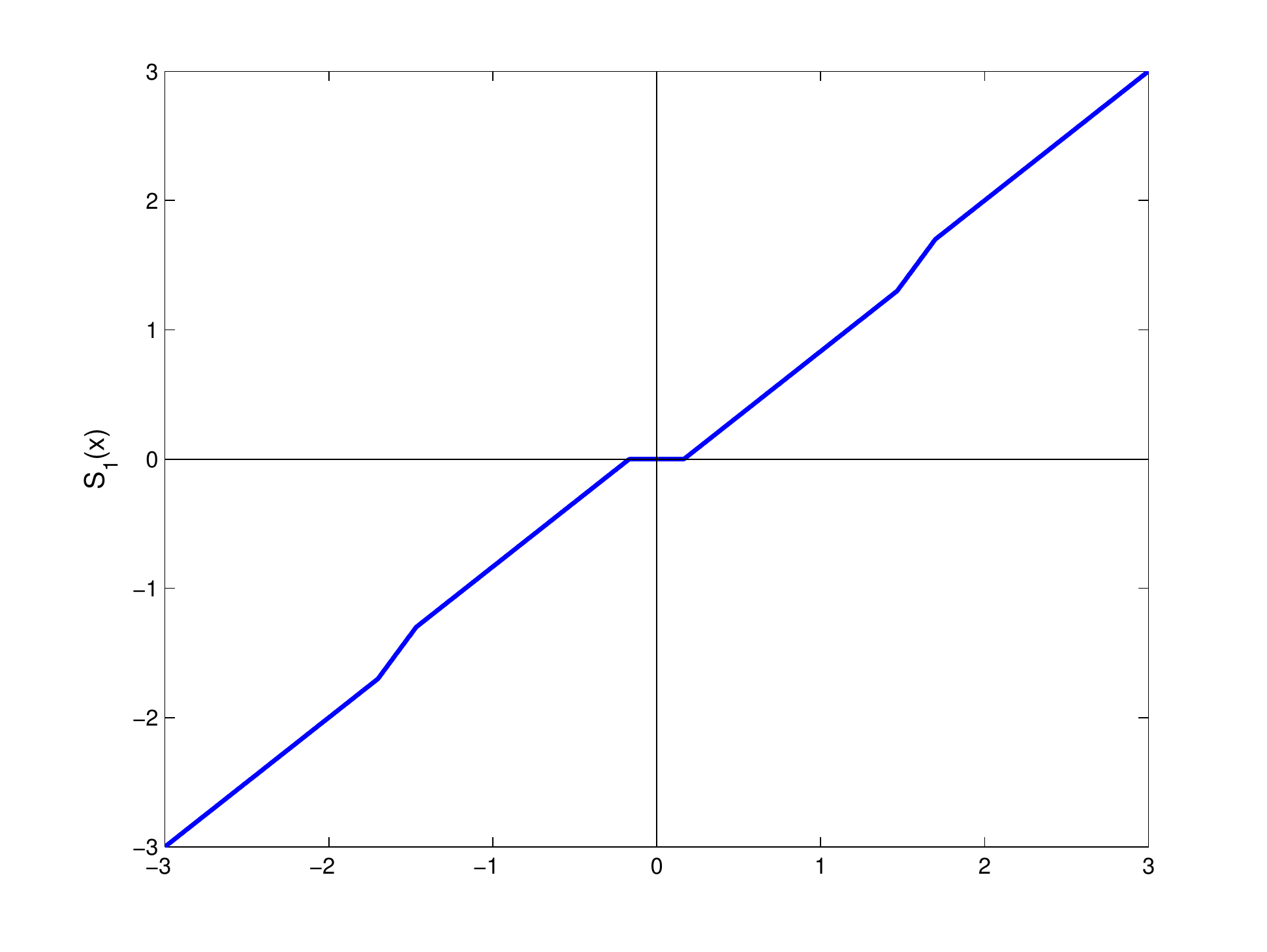}}
\subfigure[]{\label{s1(b)}
\includegraphics[width=2.3in]{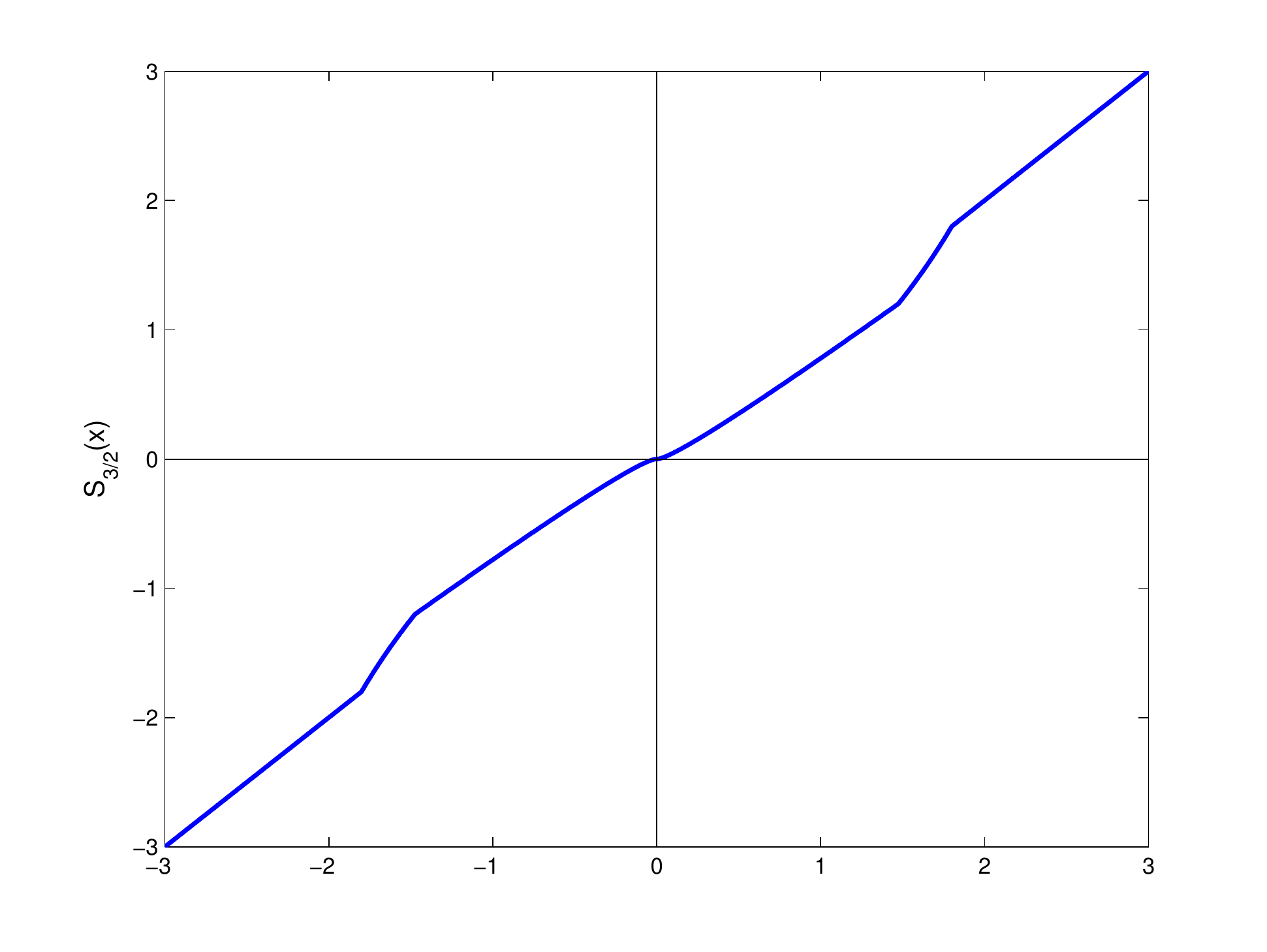}}
\subfigure[]{\label{s1(c)}
\includegraphics[width=2.3in]{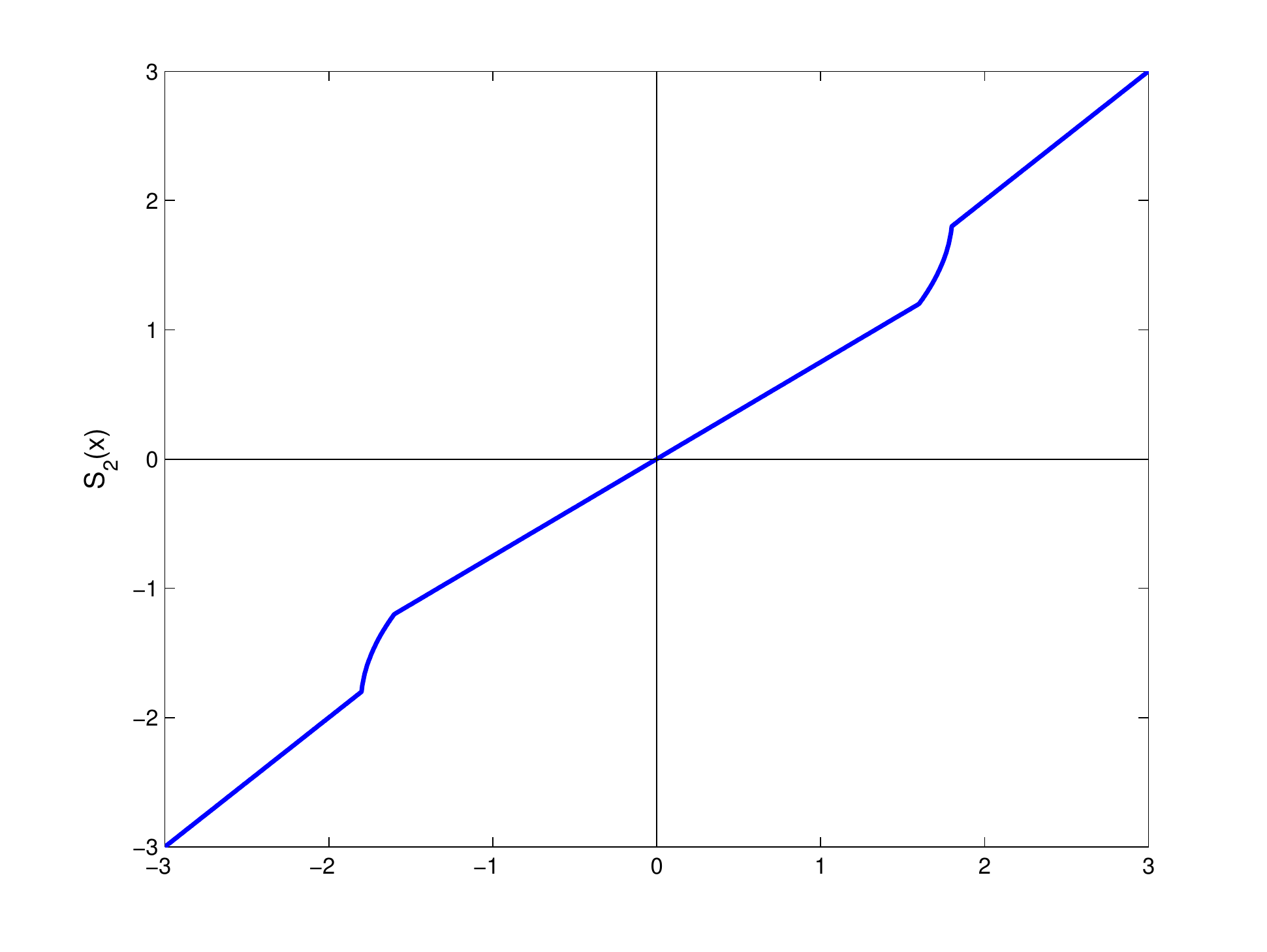}}
\end{center}
\caption{The Lipschitz continuous thresholding functions  $\mathbb S_1^\mu$,  $\mathbb S_{3/2}^\mu$, and $\mathbb S_{2}^\mu$, with parameters $p=1,3/2$, and $2$, respectively, and $r = 1.5$, $\mu = 5$, $\varepsilon=0.3$.}\label{Sthreshold}
\end{figure}

Note that in case $p\neq2$ the only part that varies is the one for $|\xi|<(\br-\epsilon)(1+\mu)$ because the remaining ones do not depend on $p$. We show in Figure \ref{Sthreshold} the typical shapes of these thresholding functions
for different choices of $p \in \{1,3/2,2\}$.
By means of these thresholding functions, the minimizing algorithm in the inner loop is given by the componentwise fixed point iteration, for $n\geq 0$,
\begin{eqnarray}
x^{n+1}_i&=&\mathbb S_p^\mu\left(\left\{\frac{1}{2}\left[(I-\frac{1}{2}A^*A)+(1-\omega)I\right]x^n+\frac{1}{2}A^*(y+q)+\omega x'\right\}_i\right), \nonumber \\
&& \phantom{XXXXXXXXXXXXXXXXXXX}i=1,\dots,N. \label{thrsal}
\end{eqnarray}
We refer to \cite[Theorem 3.17]{A.SolForn12} for the convergence properties of this algorithm.\\

To summarize, Algorithm \ref{ncbreg} can be realized in practice by nesting three loops. One external loop makes slowly vanishing the quadratic convexification, the second external loop updates the Lagrange multipliers $q_{(\ell,k)}$ for a fixed
quadratic convexification, and the final inner loop implements \eqref{thrsal}.



\section{Selectivity principle}\label{sec:selectivity}
In Section~\ref{sec:SI}, we showed a certain superiority of the regularized selective least $p$-powers (SLP) in terms of support identification and damping noise-folding. These theoretical results are supported further below in Section~\ref{sec:NR} by extensive numerical results. However, the realization of the algorithm SLP described above turns out to be computationally demanding as soon as the dimension $N$ gets large.  Since the computational time is a crucial point when it comes to practical applications, we shall introduce in this section another and more efficient approach based on the same principles to obtain equally good support identification results. The investigations in Section~\ref{sec:SI} reveal some weakness of the sparsity concept:
\\
The decoders based on $\lone$-minimization and iteratively re-weighted $\lone$-minimization prefer sparse solutions and have the undesirable effect of sparsifying also the noise. Thus all the noise may concentrate on fewer entries which, by a balancing principle, may eventually  exceed in absolute value the smallest entries of the actual original signal. This makes it impossible to separate the relevant entries of the signal from the noise by only knowing the threshold $r$ which bounds the relevant entries from below. On the contrary, SLP follows a   \emph{selectivity} principle, where the recovery process focuses on the extraction of the relevant entries, while uniformly distributing the noise elsewhere. We want now to export this latter mechanism to formulate the method described below.\\

In the following, we will show that the very well known iterative hard thresholding~\cite{blumensath2008iterative} shows similar support identification properties as SLP while being very efficient in terms of computational time. This method iteratively computes
\begin{equation}
\label{eq:IHTproc}
x^{n+1}\ldef \mathbb H_{\sqrt{\tau}}\left(x^n+A^T(y-Ax^n)\right),
\end{equation}
where
$$
(\mathbb H_{\sqrt{\tau}}(z))_i \ldef \begin{cases} z & \text{if } |z|> \sqrt{\tau}, \\ 0 & \text{else}, \end{cases},
$$
for $\tau \geq 0$. It is converging to a fixed point $x^{\text{IHT}}$ fulfilling
\begin{equation}
\label{eq:IHTFP}
x^{\text{IHT}}\ldef \mathbb H_{\sqrt{\tau}}\left(x^{\text{IHT}}+A^T(y-Ax^{\text{IHT}})\right),
\end{equation}
which is a local minimizer of the functional
$$
\mathcal J_{0}(x)\ldef\vectornorm{Ax-y}_{\ltwo}^2 + \tau \# \supp(x), 
$$
see~\cite[Theorem 3]{blumensath2008iterative} for details.
One immediately observes that iterative hard thresholding has a very selective nature: The computed fixed point $x^{\text{IHT}}$ is a vector with an unknown number, although expected to be small, of non-zero entries whose absolute value is above the threshold $\sqrt{\tau}$. In the following, we show under which conditions this method is able to exactly identify the support of the  relevant entries of the original vector $x$.
\begin{theorem}
\label{thm:SIIHT}
 Assume $A \in\R^{m\times N}$ to have the $(2k,\delta_{2k})$-RIP, with $\delta_{2k} <1$, $\vectornorm{A}\leq 1$, and define $\beta(A) > 0$ such that
\begin{equation*}
\sup\limits_{i\in\{1,\ldots,N\}} |A_i^T z| \geq \beta(A) \vectornorm{z}_{\ltwo},\quad\text{for all } z \in \R^m,
\end{equation*}
 where $A_i$ are the columns of the matrix $A$. Let $x \in \Ss_{\eta,k,r}^p$, for a fixed $1 \leq p \leq 2$, and $y = Ax$ the respective measurements. Assume further
\begin{equation}
\label{eq:RIHT}
r > \eta \left(1+\frac{1}{1-\delta_{2k}}\left(1+\frac{1}{\beta(A)}\right)\right),
\end{equation}
and define $\tau$ such that
\begin{equation}
\label{eq:lambdachoice}
\eta < \sqrt{\tau} < \frac{r-\frac{\eta}{1-\delta_{2k}}}{1+\frac{1}{(1-\delta_{2k})\beta(A)}}.
\end{equation}
Let $x^{\text{IHT}}$ be the limit of the iterative hard thresholding algorithm~\eqref{eq:IHTproc}, and thus fulfilling the fixed point equation~\eqref{eq:IHTFP}. Further, we assume
\begin{equation}
\label{eq:minIHT}
\mathcal J_0(x^{\text{IHT}}) \leq \mathcal J_0(x|_{S_r(x)}),
\end{equation}
then we have that $\Lambda\ldef S_r(x) = \supp(x^{\text{IHT}})$, and 
\begin{equation}
\label{eq:poorestim}
|x_i - x^{\text{IHT}}_i| < r-\sqrt{\tau}\text{, for all }i \in \Lambda.
\end{equation}
\end{theorem}
\begin{proof}
Assume $\#\supp({x^{\text{IHT}}}) > \# S_r(x) = k $. By~\eqref{eq:minIHT}, we have that 
\begin{align*}
0 &< \#\supp(x^{\text{IHT}}) - \#\supp(x|_{S_r(x)})  \\
&=  \#\supp(x^{\text{IHT}}) - \# S_r(x) \\
&\leq \frac{1}{\tau}\left( \vectornorm{A\left(x|_{S_r(x)}\right)-y}_{\ltwo}^2 - \vectornorm{Ax^{\text{IHT}}-y}_{\ltwo}^2\right)  \leq \frac{1}{\tau} \vectornorm{A\left(x|_{S_r(x)}\right)-y}_{\ltwo}^2 \\
&= \frac{1}{\tau} \vectornorm{A\left(x|_{(S_r(x))^c}\right)}_{\ltwo}^2 \leq \frac{1}{\tau} \vectornorm{A}^2\vectornorm{x|_{(S_r(x))^c}}_{\ltwo}^2 \leq \frac{\eta^2}{\tau} \\
& < 1, 
\end{align*} 
where the last inequality follows by~\eqref{eq:lambdachoice}. Since $(\#\supp\left(x|_{S_r(x)}\right) -\#\supp{x^{\text{IHT}}})\in\N$, the upper inequality yields to a contradiction. \\

Thus  $\#\supp(x^{\text{IHT}}) \leq \# S_r(x) = k $ and therefore $x^{\text{IHT}}$ and $x|_{S_r(x)}$ are both $k$-sparse, and $\left(x^{\text{IHT}} - x|_{S_r(x)}\right)$ is $2k$-sparse. Under our assumptions we can apply \cite[Lemma 4]{blumensath2008iterative} to obtain
\begin{equation}
\label{eq:boundIHT}
\vectornorm{A x^{\text{IHT}}-y}_{\ltwo} \leq \frac{\sqrt{\tau}}{\beta(A)}.
\end{equation}
In addition to this latter estimate, we use the RIP, the sparsity of $x^{\text{IHT}} - x|_{S_r(x)}$, and~\eqref{eq:lambdachoice} to obtain for all $i\in\{1,\ldots,N\}$ that
\begin{align*}
| \left( x|_{S_r(x)}\right)_i - x_i^{\text{IHT}}| &\leq \vectornorm{ x|_{S_r(x)} - x^{\text{IHT}}}_{\ltwo} \leq \frac{ \vectornorm{A(x|_{S_r(x)} - x^{\text{IHT}})}_{\ltwo} }{1-\delta_{2k}} \\ 
&\leq  \frac{ \vectornorm{A(x - x^{\text{IHT}})}_{\ltwo} + \vectornorm{A\left(x|_{(S_r(x))^c}\right)}_{\ltwo} }{1-\delta_{2k}} \\
& \leq \frac{ \vectornorm{y - Ax^{\text{IHT}}}_{\ltwo} + \vectornorm{A\left(x|_{(S_r(x))^c}\right)}_{\ltwo} }{1-\delta_{2k}} \\
&\leq  \frac{ \sqrt{\tau} }{\beta(A)\left(1-\delta_{2k}\right)} +  \frac{ \eta }{1-\delta_{2k}} < r-\sqrt{\tau}.
\end{align*}
Assume now that there is $\tilde{i}\in\N$ such that $\tilde{i}\in S_r(x)$ and $\tilde{i}\notin\supp(x^{\text{IHT}})$. But then we would also have $|x_{\tilde{i}}-x^{\text{IHT}}_{\tilde{i}}| = |x_{\tilde{i}}| > r $, which leads to a contradiction. Thus, we have equivalently  $S_r(x) \subset \supp(x^{\text{IHT}})$. In addition to $\#\supp(x^{\text{IHT}}) \leq \# S_r(x)$, this concludes the proof.
\end{proof}
\begin{remark} Let us discuss some of the assumptions and implications of this latter result. 
\begin{enumerate}[(i)]
\item Since iterative hard thresholding is only computing a local minimizer of $\mathcal J_0$, condition~\eqref{eq:minIHT} may not be always fulfilled for any given initial iteration $x^0$. 
Similarly to the argument in Remark \ref{remarkPP} (ii), using the $\lone$-minimizer as the starting point $x^0$, or equivalently choosing the vector $x^0$ as composed of the entries of $\Delta_1(A x)$ exceeding $\sqrt \tau$ in absolute value, we may allow us to approach a local minimizer which fulfills~\eqref{eq:minIHT}. 
\item Condition~\eqref{eq:RIHT} is comparable to the one derived in~\eqref{eq:thrs2}. If A is ``well-conditioned'', i.e., we have that $(1-\delta_{2k}) \sim 1$, and $\beta(A) \sim 1$, then
$$
1+\frac{1}{1-\delta_{2k}}\left(1+\frac{1}{\beta(A)}\right) \sim  3.
$$
\end{enumerate}
\end{remark}
Although we are able to exactly identify the support $\Lambda$ by means of Theorem~\ref{thm:SIIHT}, we only have the very poor error estimate~\eqref{eq:poorestim} of the relevant part of the signal. The reason why we cannot obtain an  estimate as good as the one in Remark~\ref{rem:precBG} is that the conditions $x^{\text{IHT}} \in \Ss_{\eta,k,r}^p$, and $Ax = Ax^{\text{IHT}}$ to apply Theorem~\ref{thm:suppid} are in general not fulfilled. Hence an additional correction to $x^{\text{IHT}}$ is necessary. As we now dispose of the support $\Lambda=S_r(x)$, a natural approach is to seek for an additional vector $x'$ which is the solution of 
\begin{align}
\nonumber\min\limits_{z\in\R^n} \quad &\vectornorm{Az-y}_{\ltwo}^2 \\
\label{eq:probCorIHT} \text{s.t.} \quad & \vectornorm{z_{\Lambda^c}}_{\ell_p}\leq \eta,  \\
 & \nonumber |z_i| \geq r \text{, for all } i\in \Lambda. 
\end{align}
Since the original signal $x$ fulfills $Ax-y = 0$, and $x \in \Ss_{\eta,k,r}^p$, it is actually a solution of problem~\eqref{eq:probCorIHT}. Thus, we conclude that for any minimizer $x'$ of problem~\eqref{eq:probCorIHT} the objective function equals zero, thus $\vectornorm{Ax'-y}_{\ell_2} = 0$, and we conclude $Ax = Ax'$ and  $x' \in \Ss_{\eta,k,r}^p$ as well. 
The optimization \eqref{eq:probCorIHT} is in general nonconvex, but, luckily, we can easily recast it in an equivalent convex one: Since $|x_i - x^{\text{IHT}}_i| < r-\sqrt{\lambda}$, and $|x_i|>r$, we know that the relevant entries of $x$ and $x^{\text{IHT}}$ have the same sign. Since we are searching for solutions which are close to $x$, the second inequality constraint becomes $\sign(x^{\text{IHT}}_i)z_i \geq r$, for all $i\in\Lambda$. 
Together with the equivalence of $\ell_2$- and $\ell_p$-norm, we rewrite problem \eqref{eq:probCorIHT} as
\begin{align}
\nonumber\min\limits_{z\in\R^n} \quad &\frac{1}{2} z^T(A^TA)z-y^TAz \\
\label{eq:probCorIHTrew} \text{s.t.} \quad & z^TP_0z - (N-k)^{1-\frac{2}{p}}\eta^2\leq 0, \\
\nonumber & z^T P_j z - (\sign(x^{\text{IHT}}_{i_j}) e_{i_j})^T z + r \leq 0 \text{, for all } i_j \in \Lambda \text{ , } j=1,\ldots,\#\Lambda,
\end{align}
where $P_0  \in \R^{N\times N}$ is defined componentwise by $$(P_0)_{r,s} \ldef \begin{cases} 1 &\text{if } r=s\in \Lambda \\ 0 &\text{else}\end{cases},$$ and $P_j = 0$, $j = 1,\ldots\#\Lambda$. Since $A^TA$, $P_0$, and $P_j$, $j=1,\ldots,\#\Lambda$, are semi-definite, problem~\eqref{eq:probCorIHTrew} is a {\it convex quadratically constrained quadratic program} (QCQP) which can be efficiently solved by standard methods, e.g., interior point methods~\cite{nowr06}. Since we combine here three very efficient methods ($\ell_1$-minimization, iterative hard thresholding, and a QCQP), the resulting procedure is much faster than the computation of SLP while, as we will show in the numerics,  keeping similar support identification properties.

\section{Numerical results}\label{sec:NR}
The following numerical simulations provide empirical confirmation of the theoretical observations in Section~\ref{sec:SI} and Section~\ref{sec:selectivity}. 
In particular, we want to show that  SLP and IHT, initialized by the $\ell_1$-minimizer as a starting value, are very robust and provide a significantly enhanced rate of recovery of the support of the unknown sparse vector as well as a better accuracy in
approximating its large entries, whenever limiting noise, i.e., $\eta \approx r$, is present on the signal.


In the previous sections we provided a very detailed description of the parameter choice for the concrete implementation of SLP and IHT, which depends in a rather explicit way on the threshold parameter $r>0$ and it is independent 
of the dimensionalities $N,m,k$ of the problem.
Instead, it turns out to be extremely hard to tune the parameter $\delta$, appearing in \cite{candesIRWL1,needellIRWL1} as the $\ltwo$-norm residual discrepancy, in order to obtain the best performances for the iteratively re-weighted $\ell_1$-minimization (IRW$\ell_1$) in terms
of  support identification and  accuracy in approximating the large entries of the original vector. In the papers \cite{candesIRWL1,needellIRWL1} the authors indicated $\delta^2=\sigma^2(m+2\sqrt{2m})$, depending on $m$, as the best parameter
choice for ameliorating the discrepancy in $\ltwo$-norm between original and decoded vector with respect to the sole $\ell_1$-minimization. However, in our  experiments we found out that for the two purposes mentioned above a much smaller $\delta$ has to be chosen. The stability parameter $a$, which avoids the denominator to be zero in the weight updating rule seems to have instead no strong influence, and it is set to 0.1 in our experiments. We executed 8 iterations of IRW$\ell_1$ as a reasonable compromise
between computational effort and accuracy.

We also consider as one of the test methods $\lone$-minimization, where we substituted the equality constraint $Ax=y$ with an inequality constraint which takes into account the noise level, folded from the noise on the signal though; thus
$$
\vectornorm{Ax-y}_{\ltwo}\leq\vectornorm{An}_{\ltwo}\leq \delta.
$$
In this constraint, we used the same parameter $\delta$ as for IRW$\ell_1$. 
\\

As we shall argue in detail below, the following numerical tests indicate that $\ell_1$+IHT is much faster and usually more robust than $\ell_1$+SLP, and that both
of them perform much better than $\ell_1$-minimization and IRW$\ell_1$ in terms of support recovery and accuracy in approximating the large entries in absolute
value of the original signal.

\paragraph{Advantages of SLP with respect to $\ell_1$-minimization}
\begin{figure}[!ht]
\begin{center}
\subfigure[]{ \label{fig:1(a)}
\includegraphics[width=1\textwidth]{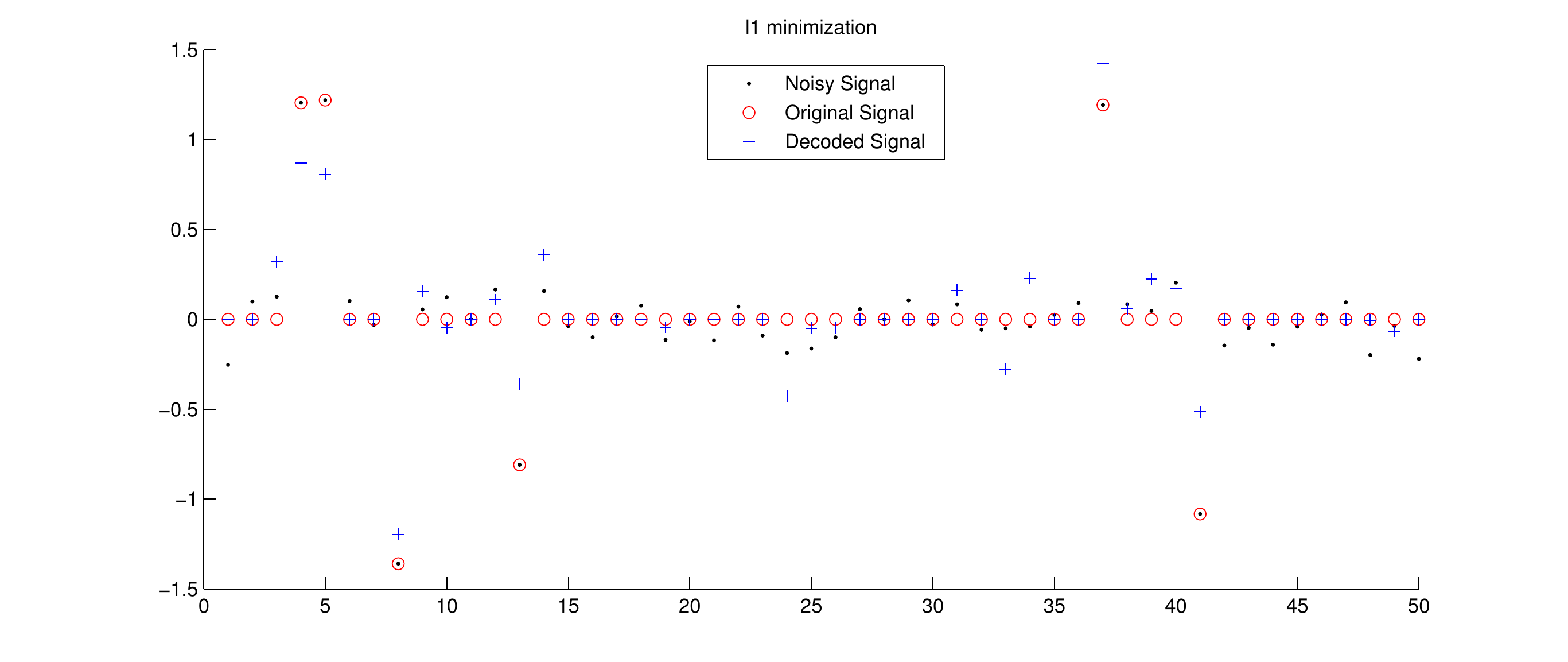}}
\subfigure[]{\label{fig:1(b)}
\includegraphics[width=1\textwidth]{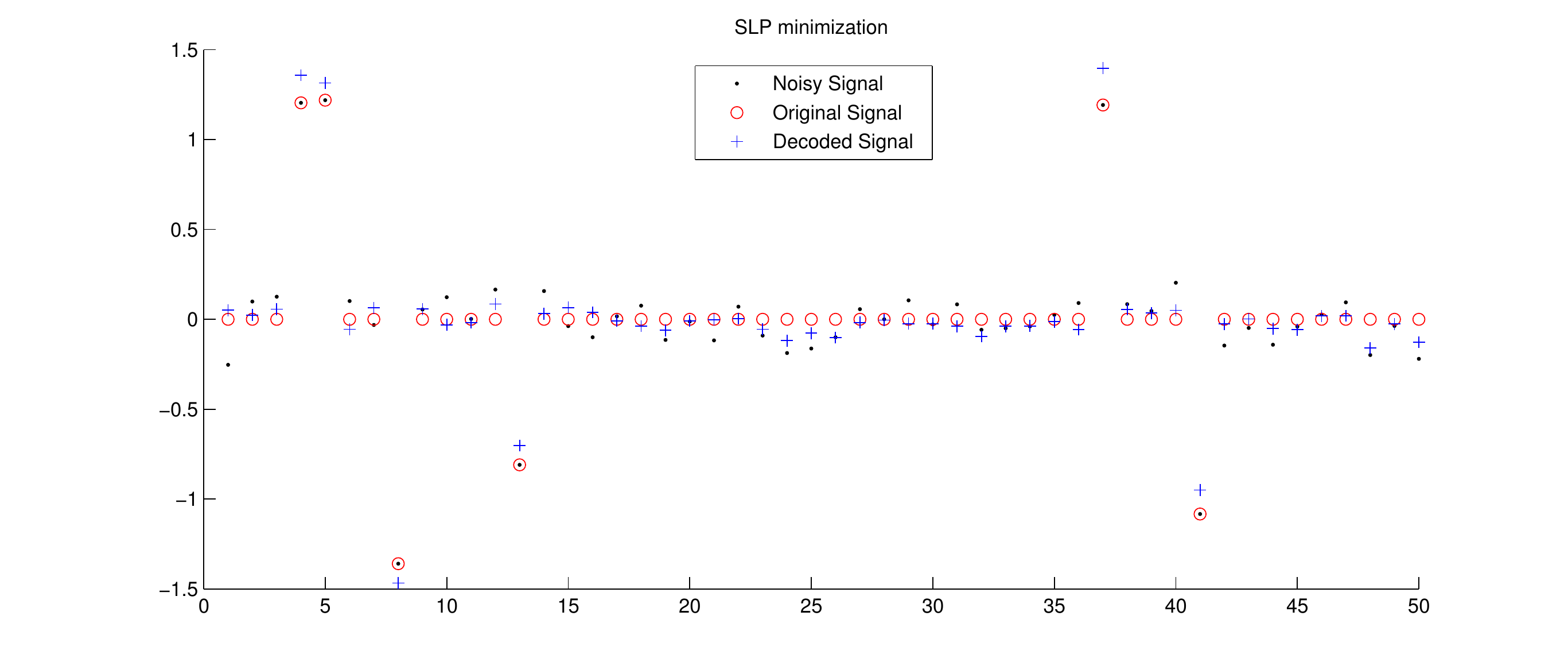}}
\end{center}
\caption{The results of the $\lone$-minimization and SLP are shown in Subfigure~\ref{fig:1(a)} and~\ref{fig:1(b)}  respectively. The two decoders are intended to recover the original signal~({\color{red}o}), starting from the measurement of the noisy signal~($\cdot$). The output of the two processes is represented by~({\color{blue}$+$}). In this case the starting value for both minimizations is $x_0=0$.}
\label{fig:l1ms}
\end{figure}

We shall start the discussion on numerical experiments with a comparison between the $\lone$-minimization and SLP for one typical example reported in Figure~\ref{fig:l1ms}. Although the setting of the two methods is the same, the results are very different: SLP-minimization can recover the signal with a very good approximation of its large entries in absolute value and a significant reduction of the noise level, while $\lone$-minimization may approximate the signal in a bad way, due to the amplification of the noise. For instance, it is evident in Subfigure \ref{fig:1(a)} that $|x_{\lone}(13)|<|x_{\lone}(24)|$ gives a wrong information about the location of the relevant entries, mismatching them with the noise. However, in this particular example, we were lucky to choose the right starting value for SLP. Due to its nonconvex character, in general SLP is computing a local minimizer, which might be far away from the original signal.

\paragraph{Choosing $\ell_1$-minimization as a warm up}

In Section \ref{sec:MSmin}, we mentioned that the algorithm, which minimizes the functional ${\SLP_r^{p,\epsilon}}$, finds only a critical point, so the condition $\SLP_r^{2,\epsilon}(x^*) \leq \SLP_r^{2,\epsilon}(x)$ \eqref{eq:minMS} used in the proof of Theorem~\ref{thm:MSsupp} may not be always valid. In order to enhance the chance of validity of this condition, the choice of an appropriate starting point is crucial. As we know that the $\lone$-minimization converges to its global minimizer with at least some guarantees given by Theorem \ref{thm:l1guar}, we use the result of this minimization process as a warm up to select the starting point of Algorithm \ref{ncbreg}. In the following, we distinguish between SLP which starts at $x_0 = 0$ and $\lone$+SLP which starts at the $\lone$-minimizer.

\begin{figure}[!htp]
\begin{center}
\includegraphics[ width=\textwidth]{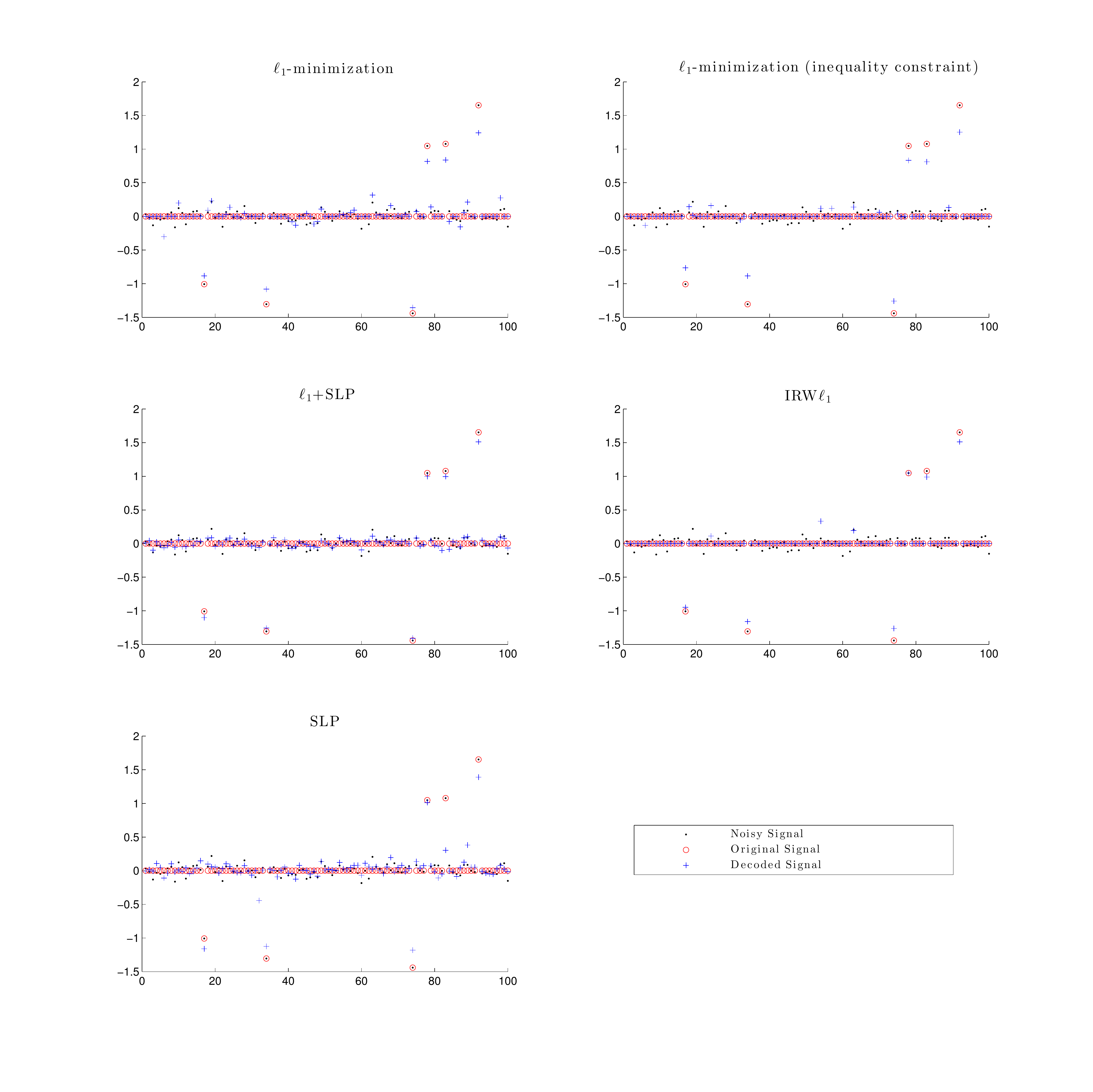}
\caption{The figure reports the results of five different decoding processes ({\color{blue}$+$}) of the same problem where the circles ({\color{red}o}) represent the original signal and the points ($\cdot$) represent the original signal corrupted by the noise.}
\label{good-bad}
\end{center}
\end{figure}

In Figure~\ref{good-bad}, we illustrate the robustness of $\ell_1$+SLP (bottom left subfigure) in comparison to the $\ell_1$-minimization based methods and SLP starting at $0$. Here, SLP converged to a feasible critical point, but it is quite evident that the decoding process failed since the large entry at position 83 (signal) was badly recovered and even the entry at position 89 (noise) is larger. If we look at the $\lone$-minimization result (top left subfigure) or the $\lone$-minimization with inequality constraint (top right subfigure), the minimization process brings us close to the solution, but the results still significantly lack accuracy. By $\lone$+SLP (center left subfigure) we obtain a good approximation of the relevant entries of the original signal and we get a significant correction and an improved recovery. Also IRW$\lone$ improves the result of $\ell_1$-minimization significantly, but still approximates the large entries worse than $\lone$+SLP. Although the difference is minor, we observe another important aspect of IRW$\lone$: the noise part is sparsely recovered, while $\lone$+SLP distributes the noise in a more uniform way in a much smaller stripe around zero. This drawback of IRW$\lone$ can be crucial when it comes to the distinction of the relevant entries from noise. 

\paragraph{Massive computations}
The previously presented specific examples in support of our new decoding strategies are actually typical. 
In order to support this work with even more impressive and convincing evidences, we present some statistical data obtained by solving series of problems. We decided to fix the parameters in order to have the most coherent data to be analyzed; in particular, we set $N=100$, $m=40$, $r=0.8$, $k=1,\ldots,7$, and $\eta = 0.75$. The vector $n$ is composed of random entries with normal distribution and then it is rescaled in order to have $\vectornorm{n}_{\ell_2} = \eta$. Figures~\ref{gerr}, \ref{gbig}, and \ref{gsupp} report the results obtained considering 30 different i.i.d.Gaussian encoding matrices while Figures~\ref{ferr},~\ref{fbig}, and \ref{fsupp} used 30 random subsampled cosine transformation encoding matrices. In the following, we use $x^*$ generically for the decoded vector of any method.\\

\begin{figure}[!ht]
\begin{center}
\includegraphics[width=\textwidth]{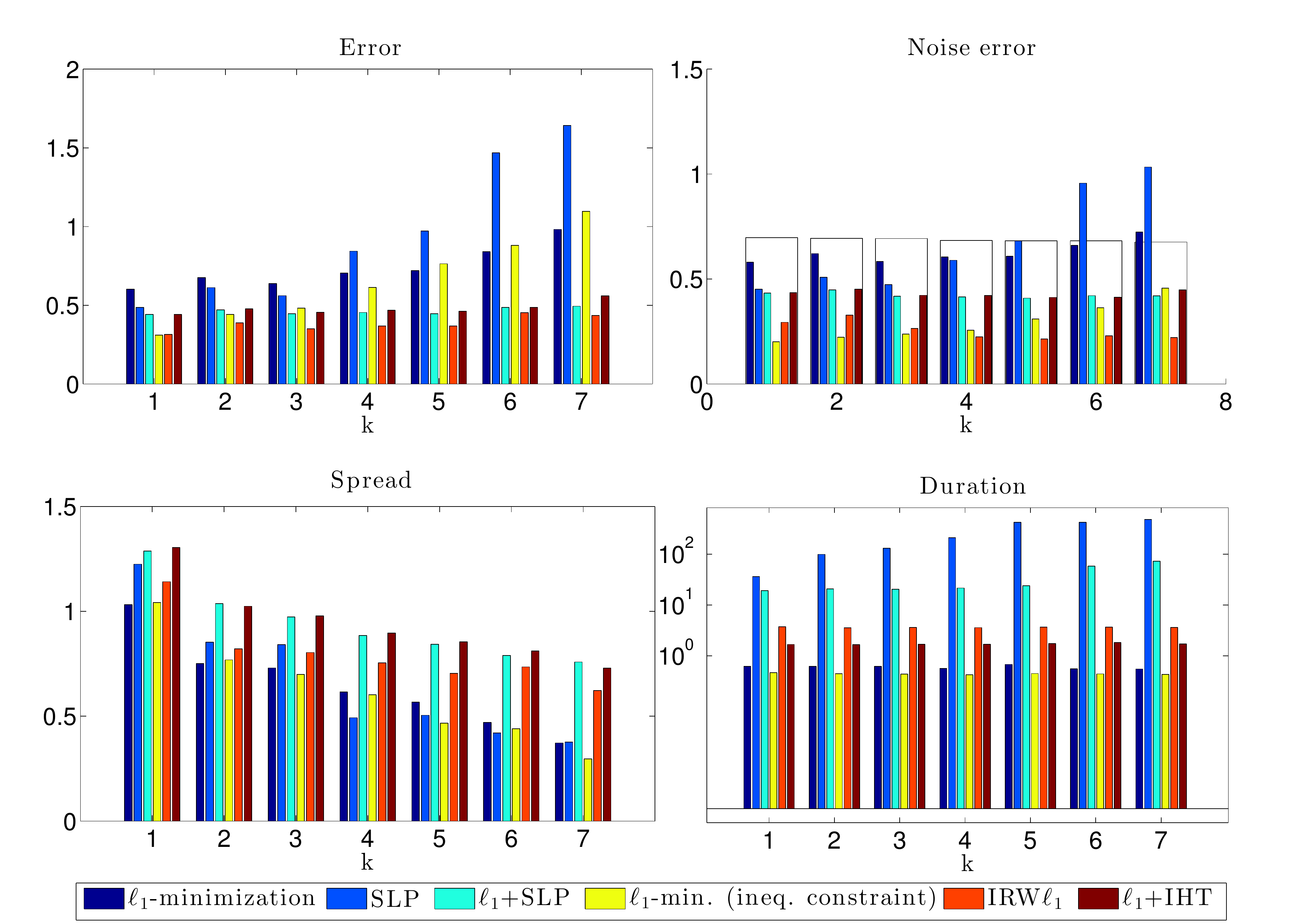}
\caption{The columns refer to the different results of $\lone$-minimization ({\color{MidnightBlue}dark blue}), SLP ({\color{blue}blue}), $\lone$+SLP ({\color{cyan}cyan}), $\lone$-minimization with inequality constraint ({\color{yellow}yellow}), IRW$\lone$ ({\color{orange}orange}), and $\ell_1$+IHT ({\color{BrickRed}brown}). In the \emph{Noise error} subfigure the white column in the background represents the noise level. On the x-axis the different values of $k$ are displayed and each column is the mean of the results given by 30 trials. The results were obtained by Gaussian matrices.}
\label{gerr}
\end{center}
\end{figure}

Figure ~\ref{gerr} reports the first part of the statistics which we have collected for the Gaussian matrices. We start commenting the subfigures clockwise. The first subfigure, on the upper-left, represents the mean value of the error between the exact signal and the decoded one $\vectornorm{x-x^*}_{\ltwo}$. For $\lone$+SLP, $\lone$+IHT, and \irwl the absolute $\ell_2$-norm discrepancy between original and decoded vector seems to be stable and independent of the number $k$ of large entries. These methods outperform $\lone$-minimization; \irwl performs slightly better. We also observe that the choice of the starting point is crucial for SLP and IHT. 

The second subfigure is the mean value of the noise level 
and we can see exactly what we inferred looking at Figure~\ref{good-bad}: ${\lone}$-minimization returns a larger noise level with respect to all the other methods, except SLP; and  IRW$\lone$ has the best noise reduction property. 

The third is the mean computational time, presented in logarithmic scale. All tests were implemented and run in Matlab R2013b in combination with CVX~\cite{cvx,gb08}, to solve the $\ell_1$-minimization with equality and inequality constraint, its iteratively re-weighted version, and the QCQP. We observe that SLP and $\ell_1$+SLP are extremely slow. However, in comparison, the good starting point for SLP provides an advantage in terms of computational time. IHT has a computational complexity in between $\ell_1$-minimization and IRW$\lone$.

The fourth plot reports the mean value of the discrepancy between noise level and the large entries of the signal, thus
$\min\limits_{i\in S_r(x)} |x_i^*| - \max\limits_{i\in S_r(x)^c} |x_i^*|$. This plot shows how good the small entries are distinguished from the large ones in absolute value. We see that $\ell_1$+SLP and $\ell_1$+IHT perform best, which again is a result of their non-sparse noise recovery.  \\

\begin{figure}[!htb]
\begin{center}
\includegraphics[width=\textwidth]{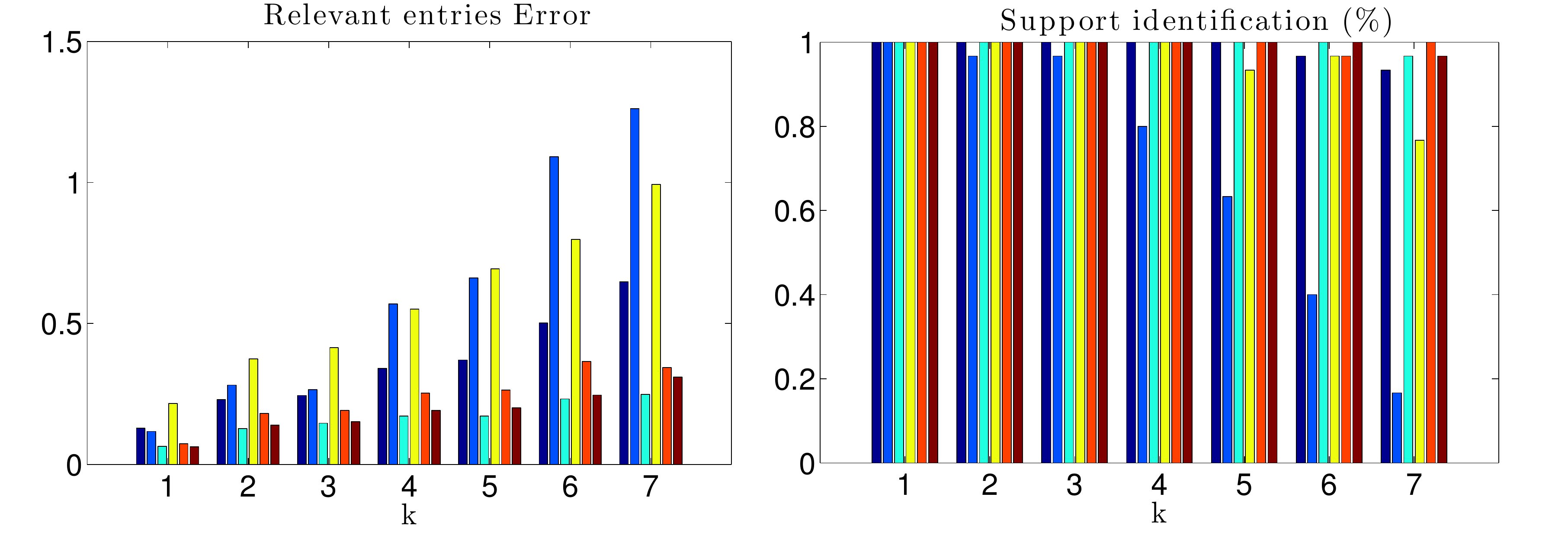}
\end{center}
\caption{The subfigures represent the error on the relevant entries and the support identification property by knowledge of $k$. For more details on the displayed data we refer to the caption of Figure~\ref{gerr}. The results were obtained by Gaussian matrices.}
\label{gbig}
\end{figure}

In Figure~\ref{gbig}, we report the histogram of the mean-value of the errors on the relevant entries: the quantities on the left subfigure are computed as the mean values of $\vectornorm{x_{S_r(x)}-x_{S_r(x)}^*}_{\ltwo}$ where we suppose to know the $k$ largest entries of the original signal. The right subfigure shows how often the $k$ largest entries of $x^*$ coincided with $S_r(x)$. Notice that there might be entries below the threshold $r$ among the $k$ largest entries of $x^*$. We conclude that, knowing the number of large entries, IRW$\lone$, $\lone$-minimization, $\lone$+SLP, and $\lone$+IHT recover the support with nearly 100\% success. In addition, $\lone$+SLP approximates best the magnitudes of the relevant entries. \\

\begin{figure}[!ht]
\begin{center}
\includegraphics[width=\textwidth]{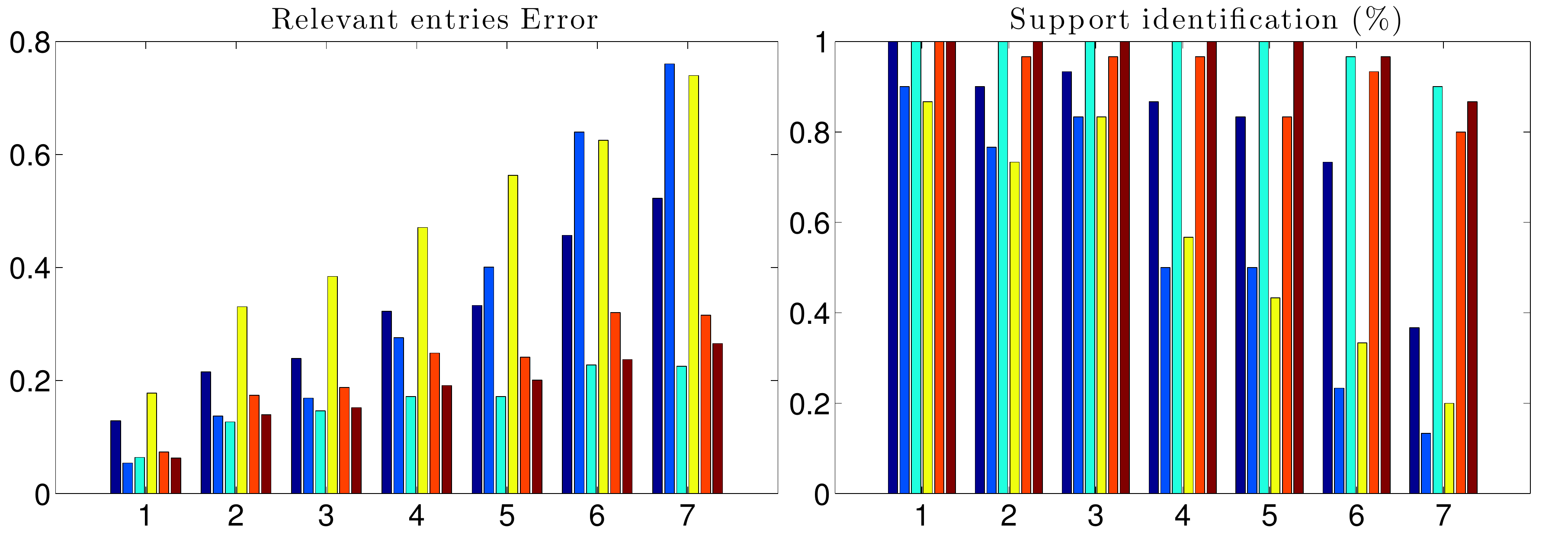}
\end{center}
\caption{The subfigures represent the error on the relevant entries and the support identification property by knowledge of $r$. For more details on the displayed data we refer to the caption of Figure~\ref{gerr}. The results were obtained by Gaussian matrices.}
\label{gsupp}
\end{figure}

In Figure \ref{gsupp} we compute again the mean-value of the relevant entries, but this time without the knowledge of $k$ but the knowledge of $r$ and therefore $S_r(x^*)$: the quantities on the left subfigure are the mean values of  $\vectornorm{x|_{S_r(x^*)}-x^*|_{S_r(x^*)}}_{\ltwo}$. In the right subfigure we attribute a positive match in case $S_r(x^*)=S_r(x)$ so that the relevant entries of $x^*$ coincide with the ones of the original signal. By our theory, we expect $\lone$+SLP and $\lone$+IHT to produce a high rate of success of correctly recovered support. Actually this is confirmed by the experiments:  Both methods do a very accurate recovery, as it gives us almost always $100\%$ of the correct result while the other methods perform worse. 

Figures~\ref{ferr},~\ref{fbig}, and \ref{fsupp} represent the same statistical data reported respectively in the Figures~\ref{gerr}, \ref{gbig}, and \ref{gsupp} but using random subsampled cosine transformation encoding matrices. Without describing the result in detail, we state that they are very similar for these problems as well.
\begin{figure}[!htb]
\begin{center}
\includegraphics[width=\textwidth]{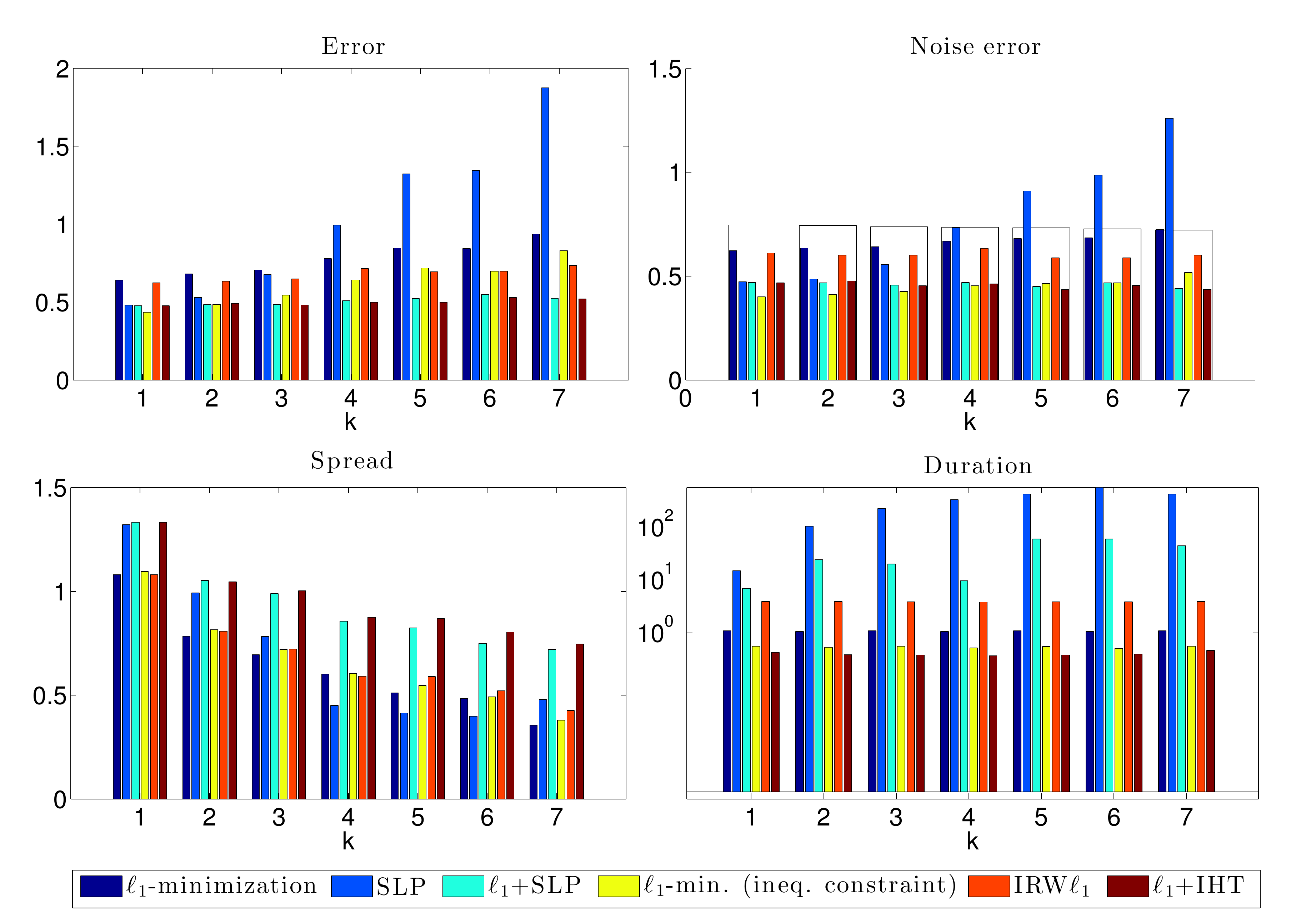}
\caption{The subfigures represent the errors, the computation time, and the separation of the component. For more details on the displayed data we refer to the caption of Figure~\ref{gerr}. The results were obtained by randomly subsampled cosine transformation encoding matrices.}
\label{ferr}
\end{center}
\end{figure}
\begin{figure}[!ht]
\begin{center}
\includegraphics[width=\textwidth]{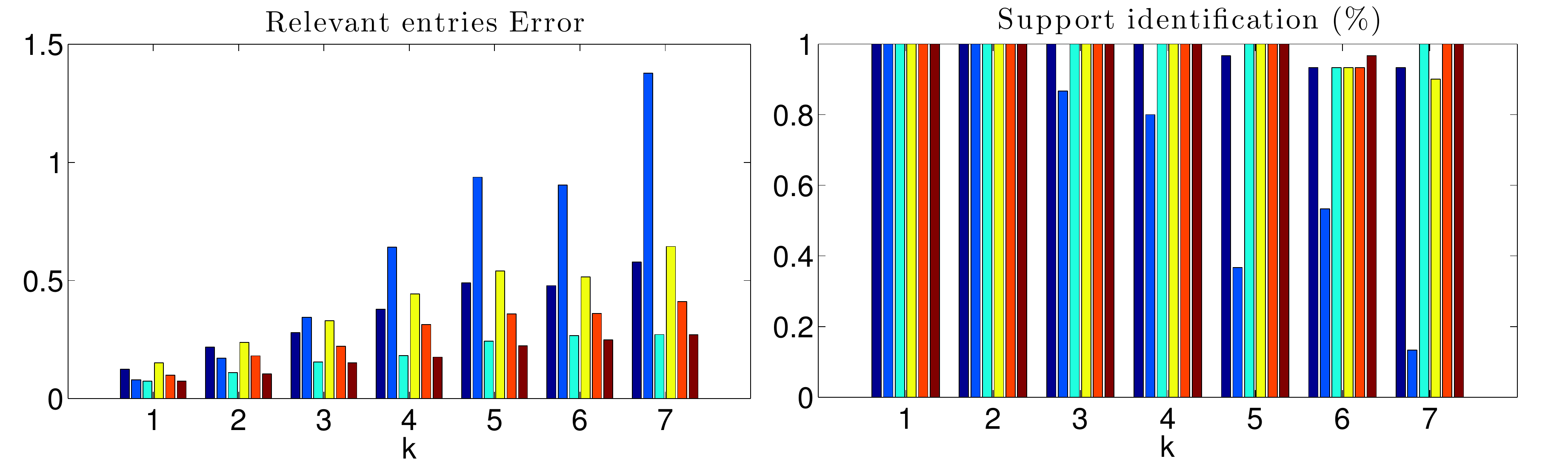}
\end{center}
\caption{The subfigures represent the error on the relevant entries and the support identification property by knowledge of $k$. For more details on the displayed data we refer to the caption of Figure~\ref{gerr}. The results were obtained by random subsampled cosine transformation encoding matrices.}
\label{fbig}
\end{figure}

\begin{figure}[!ht]
\begin{center}
\includegraphics[width=\textwidth]{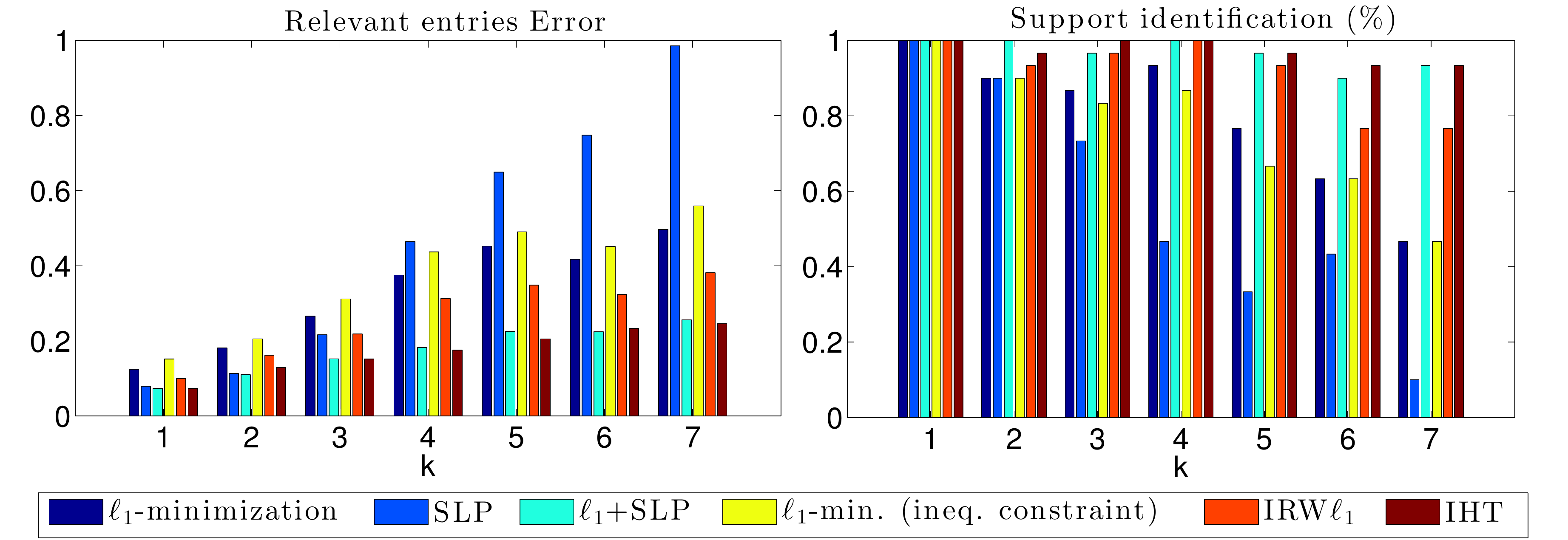}
\end{center}
\caption{The subfigures represent the error on the relevant entries and the support identification property by knowledge of $r$. For more details on the displayed data we refer to the caption of Figure~\ref{gerr}. The results were obtained by random subsampled cosine transformation encoding matrices.}
\label{fsupp}
\end{figure}

\begin{figure}[p!]
\begin{center}

\hspace{-1cm}
\subfigure[]{ \label{dis(a)}
\includegraphics[width=0.51\textwidth]{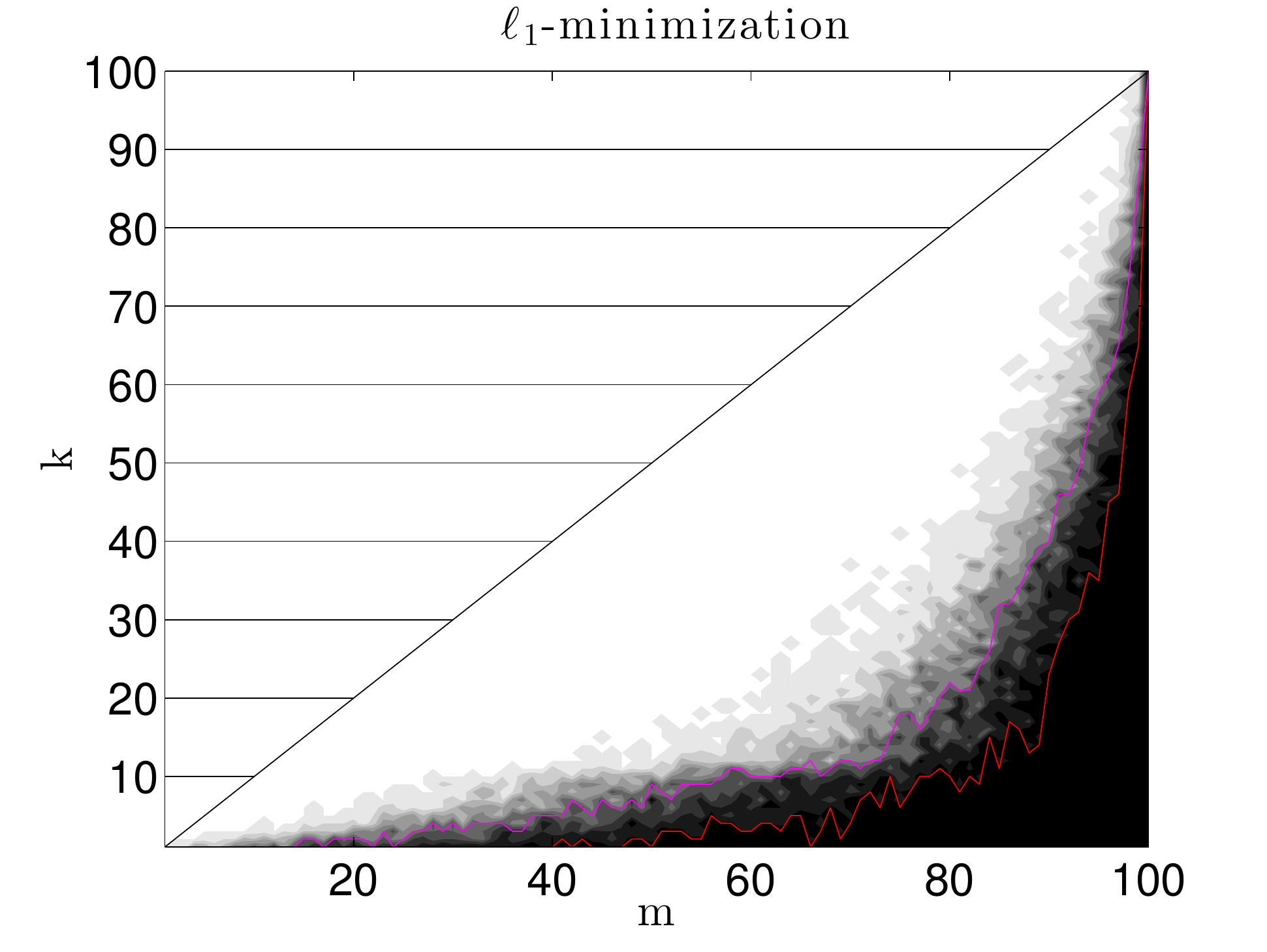}} 
\subfigure[]{\label{dis(b)}
\includegraphics[width=0.51\textwidth]{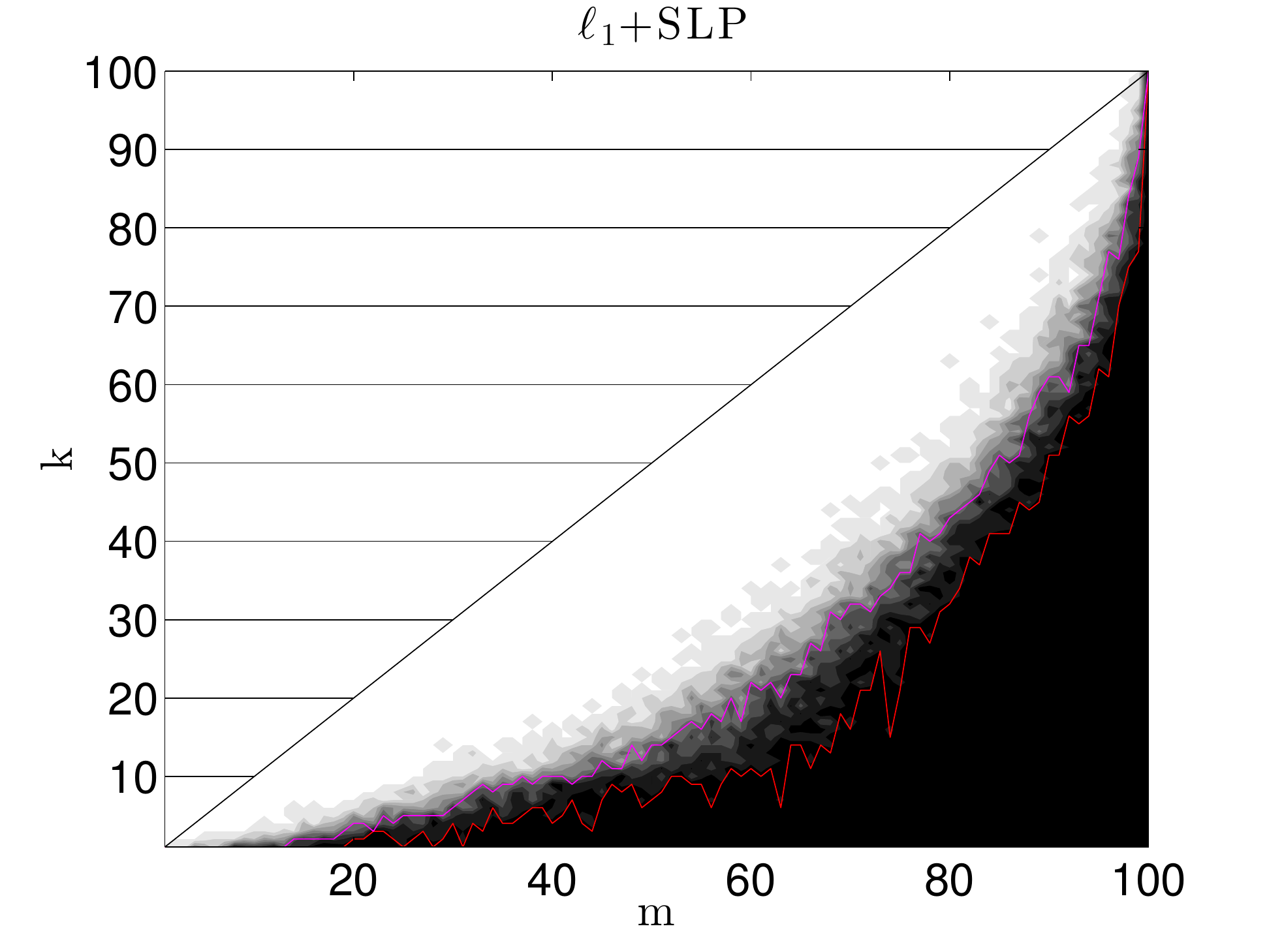}} 
~\\
\hspace{-1cm}
\subfigure[]{\label{dis(c)}
\includegraphics[width=0.51\textwidth]{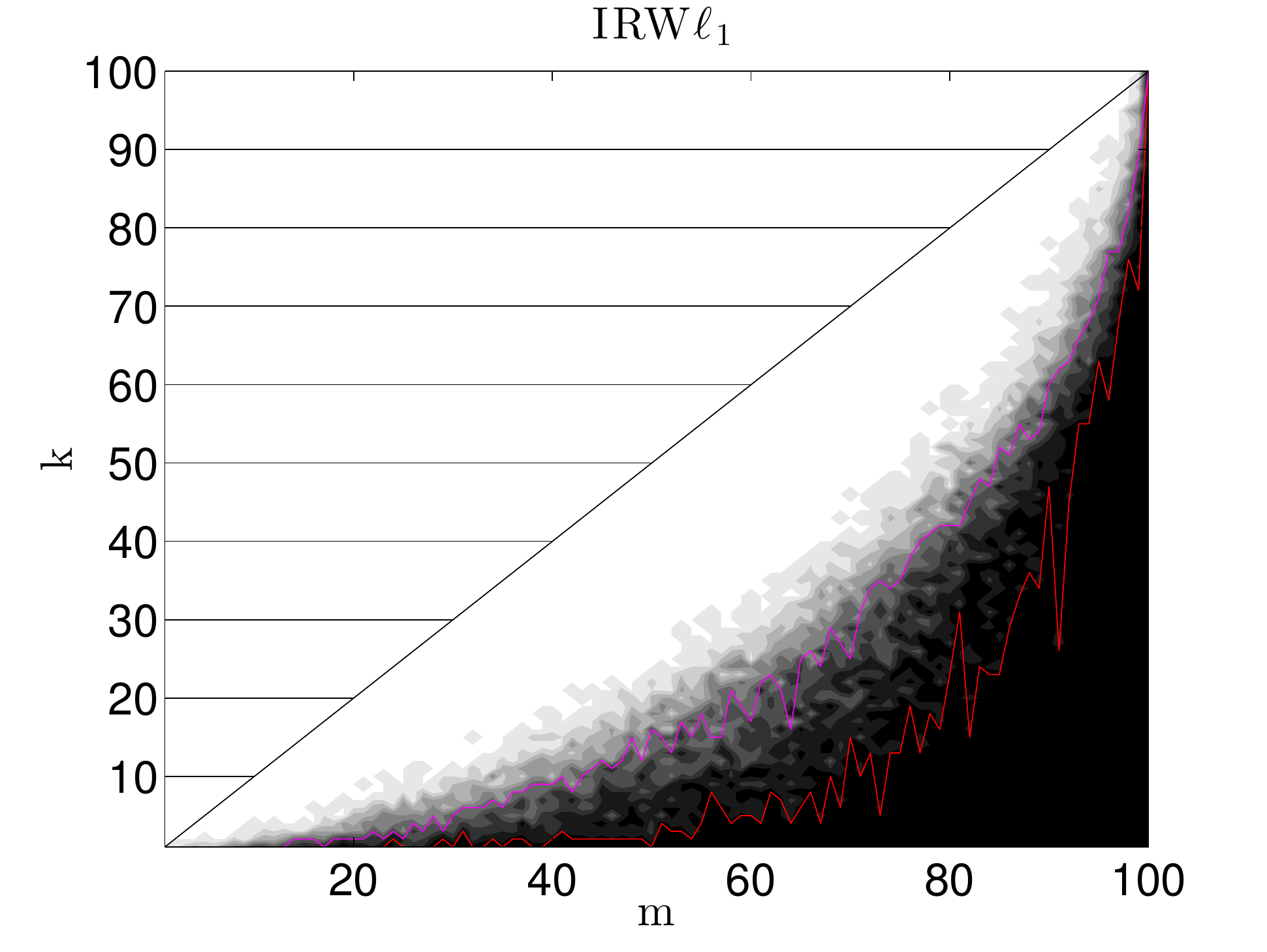}}
\subfigure[]{\label{dis(d)}
\includegraphics[width=0.51\textwidth]{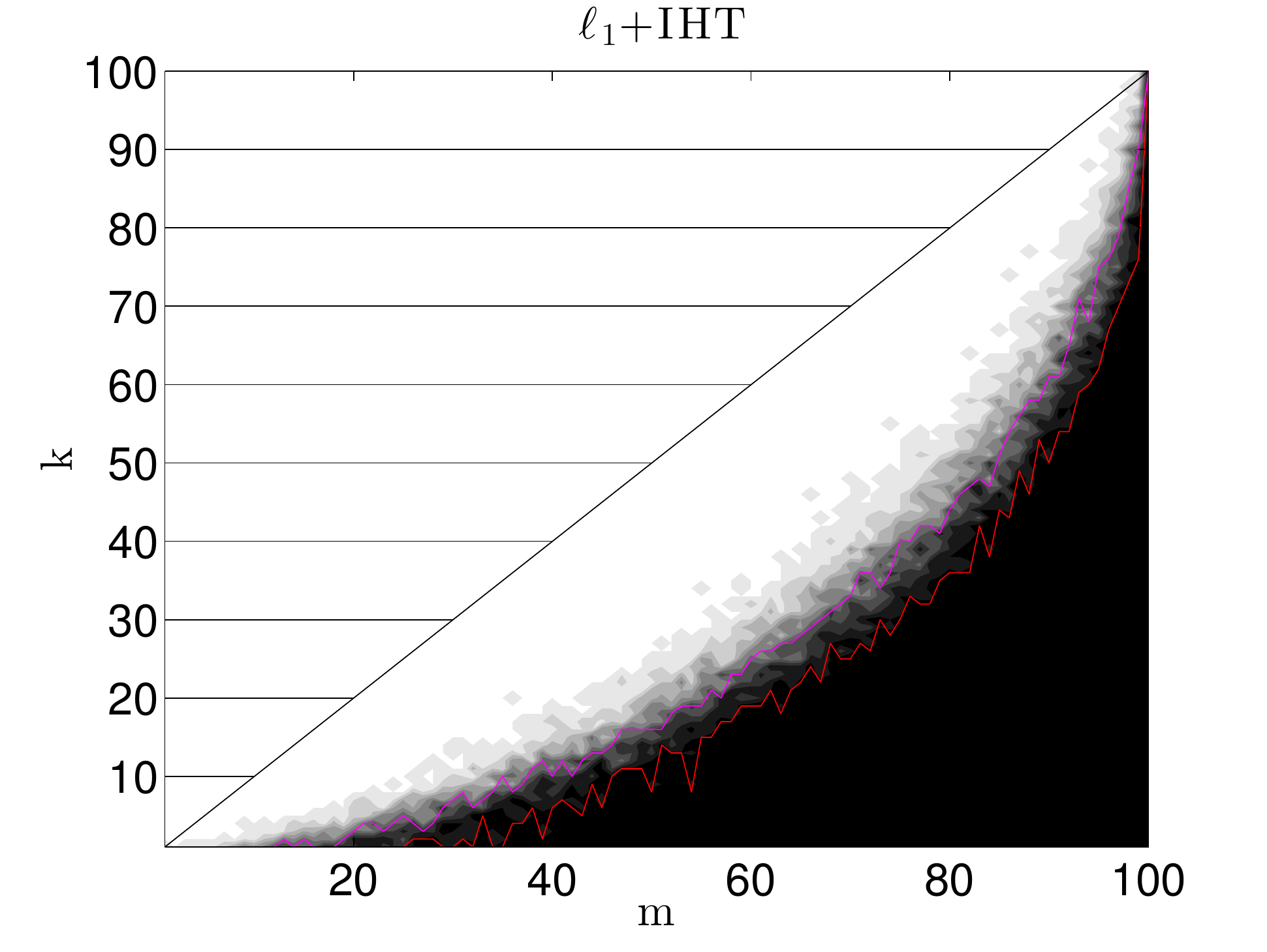}}
\end{center}
\caption{Phase transition diagrams. The black area represents the couple $(m,k)$ for which we had 100\% of support recovery. The results of (a) $\lone$-minimization, (b) $\lone+$SLP, (c) IRW$\lone$, and (d) $\lone$+IHT are reported. Note that the area for $k>m$ is not admissible. The {\color{red}red} line shows the level bound of 90\% of support recovery, and the {\color{magenta}magenta} line 50\% respectively.}

\label{dis}
\end{figure}

\begin{figure}[!htb]
\begin{center}
\subfigure[]{ \label{PTD_comp_50}
\includegraphics[width=0.48\textwidth]{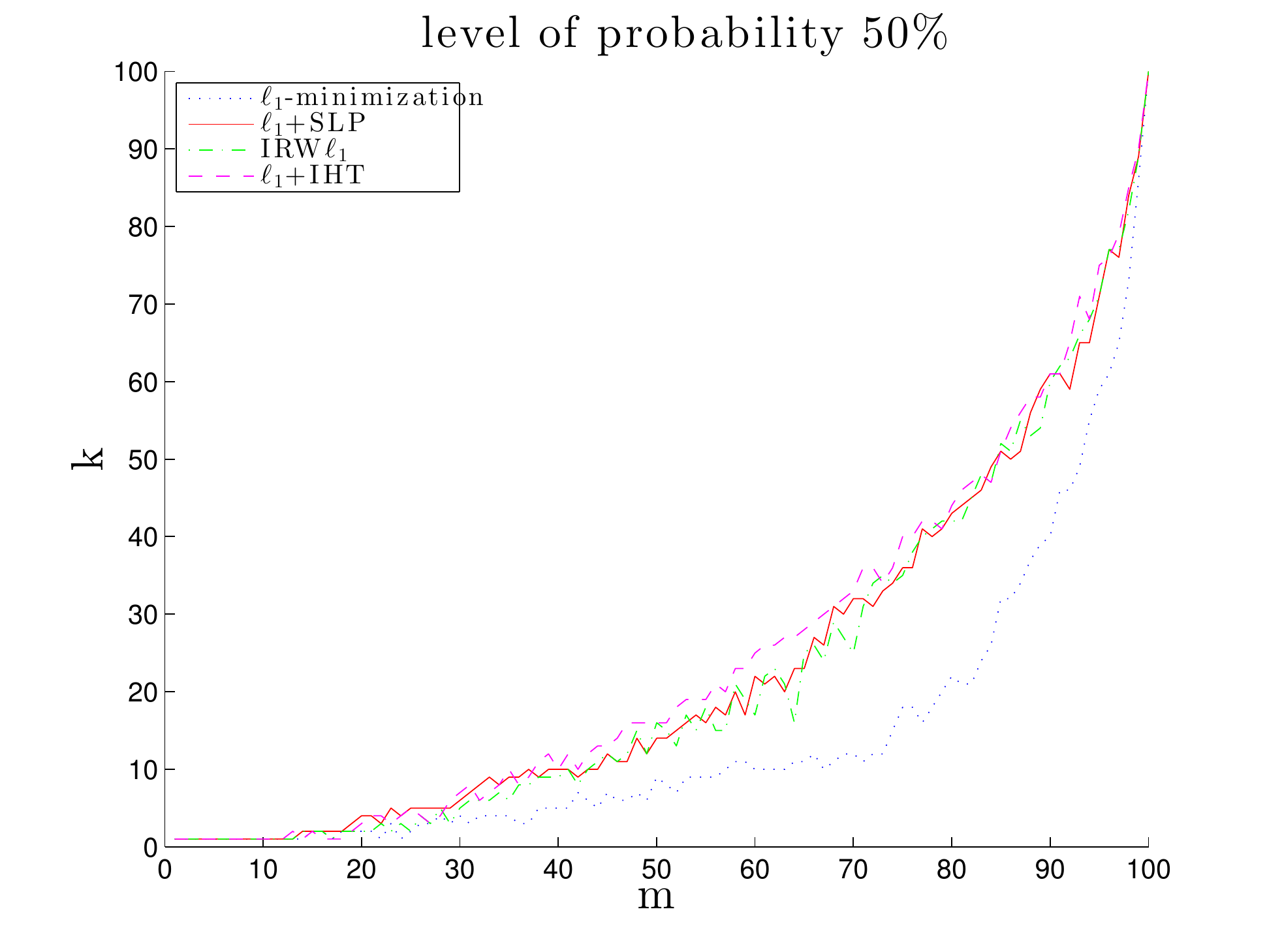}}
\subfigure[]{\label{PTD_comp_90}
\includegraphics[width=0.48\textwidth]{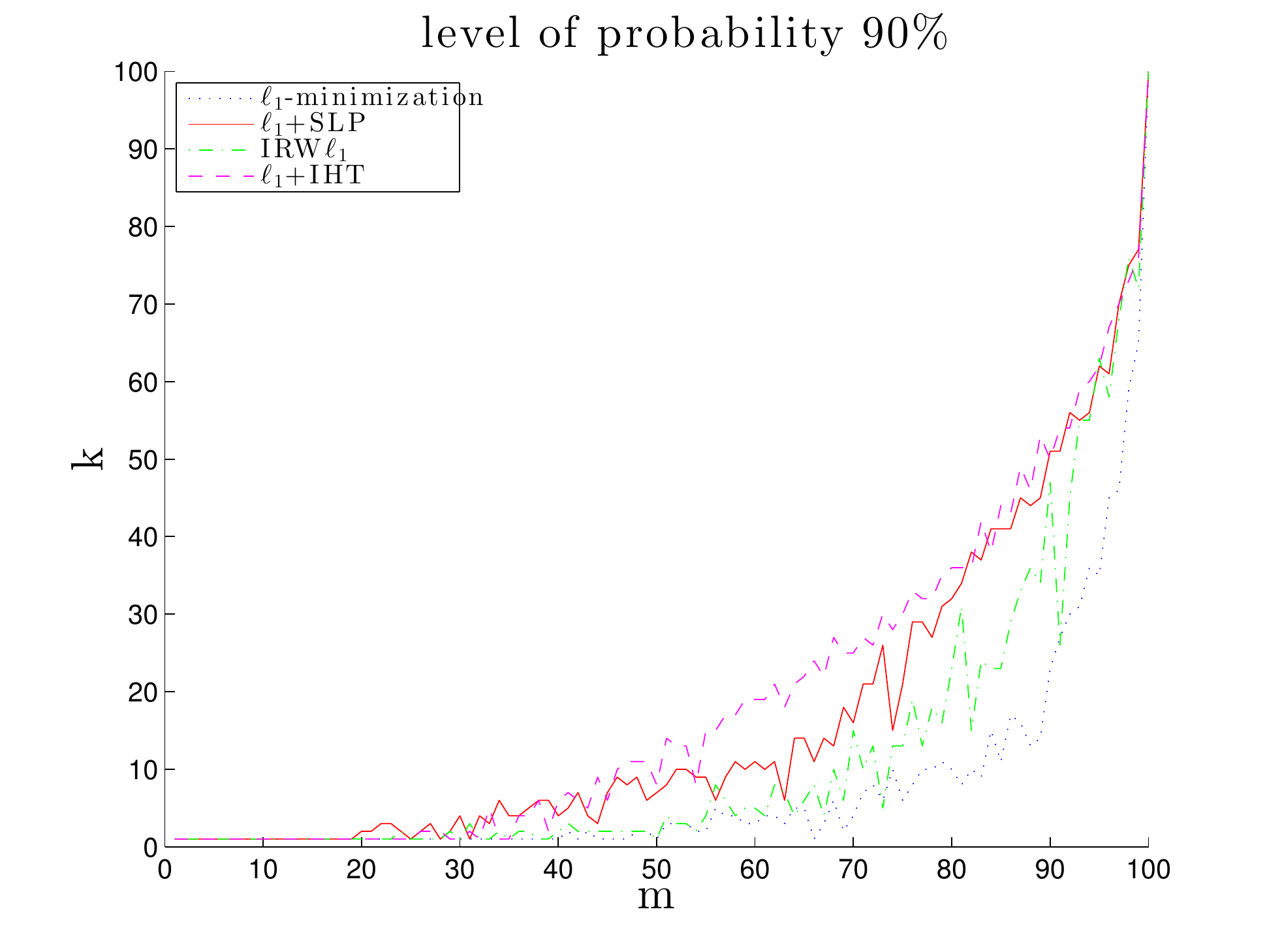}}
\end{center}
\caption{Comparison of phase transition diagrams for $\lone$-minimization ({\color{MidnightBlue}dark blue, dotted}), $\lone$+SLP ({\color{red}red}), IRW$\lone$ ({\color{green}green, dash-dotted}), and $\lone$+IHT ({\color{magenta}magenta, dashed}). The level bound of 50\% and 90\% as it is displayed in Figure~\ref{dis} is compared in (a) and (b) respectively.}
\label{fig:PTD_comp}
\end{figure}

\paragraph{Phase transition diagrams}
To give an even stronger support of the results in the previous paragraph, we extended the results of Figure~\ref{gsupp} to a wider range of $m$ and $k$. In Figure~\ref{dis}, we present phase transition diagrams of success rates in support recovery for $\lone$-minimization, IRW$\lone$, $\lone+$SLP, and $\lone$+IHT in presence of nearly maximally allowed noise, i.e., $0.8=r > \eta = 0.75$.

To produce phase transition diagrams, we varied the dimension of the measurement vector $m=1,\ldots,N$ with $N=100$, and solved 20 different problems for all the admissible $k=\#S_r(x)=1,\ldots,m$. We colored black all the points $(m,k)$, with $k\leq m$, which reported 100\% of correct support identification, and gradually we reduce the tone up to white for the 0\% result. The level bound of 50\% and 90\% is highlighted by a magenta and red line respectively. A visual comparison of the corresponding phase transitions confirms our previous expectations. In particular, $\lone$+SLP and $\lone$+IHT very significantly outperform $\lone$-minimization in terms of correct support recovery. The difference of both methods towards IRW$\lone$ is less significant but still important. In Figure~\ref{fig:PTD_comp}, we compare the level bounds of 50\% and 90\% among the four different methods. Observe that the 90\% probability bound indicates the largest positive region for $\lone$+IHT, followed by $\lone$+SLP, and only eventually by IRW$\lone$, while the bounds are much closer to each other in the case of the 50\% bound. Thus, surprisingly, $\ell_1$+IHT works in practice even better than $\lone$+SLP for some range of $m$, and offers the most stable support recovery results.


 \section*{Funding}
Marco Artina acknowledges the support of the International Research Training Group IGDK 1754 ``Optimization and Numerical Analysis for Partial Differential Equations with Nonsmooth Structures'' of the German Science Foundation.
Massimo Fornasier  acknowledges the support of the ERC-Starting Grant HDSPCONTR ``High-Dimensional Sparse Optimal Control''. Steffen Peter acknowledges the support of the Project ``SparsEO: Exploiting the Sparsity in Remote Sensing for Earth Observation'' funded by Munich Aerospace.

\bibliographystyle{abbrv}
\bibliography{damp_noise_folding_techreport}

\end{document}